\documentclass[12pt]{amsart}
\usepackage{graphicx}
\usepackage[active]{srcltx}
\usepackage{amsthm,mathrsfs}
\usepackage{a4wide}
\usepackage{amsfonts}
\usepackage{amsmath}
\usepackage{amssymb}
\newtheorem{lemma}{Lemma}
\newtheorem{theorem}{Theorem}
\newtheorem{corollary}{Corollary}
\newtheorem{remark}{Remark}
\newtheorem{example}{Example}
\newtheorem{proposition}{Proposition}

\theoremstyle{definition}
\newtheorem{definition}[theorem]{Definition}
\usepackage{float}
\usepackage{url}
\usepackage{enumitem}
\usepackage{epstopdf}

\usepackage{ mathrsfs }

\usepackage{calligra}
\usepackage{dsfont}

\newcommand {\PP} {\mathbb{P}}
\newcommand {\ZZ} {\mathbb{Z}}

\newcommand {\NN} {\mathbb{N}}

\newcommand {\FF} {\mathbb{F}}

\makeatletter\def\blfootnote{\xdef\@thefnmark{}\@footnotetext}\makeatother
\allowdisplaybreaks
\title{\bf Discrepancy bounds for normal numbers generated by necklaces in arbitrary base}

\author{Roswitha Hofer and Gerhard Larcher} 
\address{Institute of Financial Mathematics and Applied Number Theory, Johannes Kepler University Linz, Altenbergerstra{\ss}e 69, 4040 Linz, AUSTRIA}
\email{roswitha.hofer@jku.at, gerhard.larcher@jku.at}

\thanks{The authors are supported by the Austrian Science Fund (FWF): Projects F5505-N26 and F5507-N26, which are part of the Special Research Program ``Quasi-Monte Carlo Methods: Theory and Applications''}

\begin{document}

\begin{abstract}
Mordechay B. Levin in \cite{Lev} has constructed a number $\lambda$ which is normal in base 2, and such that the sequence $(\left\{2^n \lambda\right\})_{n=0,1,2,\ldots}$ has very small discrepancy $D_N$. Indeed we have $N\cdot D_N = \mathcal{O} \left(\left(\log N\right)^2\right)$. This construction technique of Levin was generalized in  \cite{Bec}, where the authors generated normal numbers via perfect nested necklaces, and where they showed that for these normal numbers the same upper discrepancy estimate holds as for the special example of Levin. In this paper now we derive an upper discrepancy bound for so-called semi-perfect nested necklaces and show that for the Levin's normal number in arbitrary prime base $p$ this upper bound for the discrepancy is best possible, i.e., $N\cdot D_N \geq c\left(\log N\right)^2$ with $c>0$ for infinitely many $N$. This result generalizes \cite{Lar} where we ensured for the special example of Levin for the base $p=2$, that $N\cdot D_N =O( \left(\log N\right)^2)$ is best possible in $N$. So far it is known by a celebrated result of Schmidt \cite{Schm} that for any sequence in $[0,1)$, $N\cdot D_N\geq c \log N$ with $c>0$ for infinitely many $N$. So there is a gap of a $\log N$ factor in the question, what is the best order for the discrepancy in $N$ that can be achieved for a normal number. Our result for Levin's normal number in any prime base on the one hand might support the guess that $O( \left(\log N\right)^2)$ is the best order in $N$ that can be achieved by a normal number, while generalizing the class of known normal numbers by introducing e.g. semi-perfect necklaces on the other hand might help for the search of normal numbers that satisfy smaller discrepancy bounds in $N$ than $N\cdot D_N=O( \left(\log N\right)^2)$. 
\end{abstract}

\date{}
\maketitle

\section{Introduction}
Let $b \ge 2$ be an integer. 
A real number $\alpha \in [0,1)$ is called \emph{normal in base $b$} if in its base $b$ representation $\alpha = 0. \alpha_1 \alpha_2 \ldots$ the following holds: For every positive integer $k$ and any block $a_1 a_2 \ldots a_k \in \left\{0,1, \dots , b-1\right\}^k$ of length $k$ we have
$$
\lim_{N \rightarrow \infty} \frac{1}{N} \# \left\{0 \leq n < N \left|\right. \alpha_{n+1} \alpha_{n+2} \ldots \alpha_{n+k} = a_1 a_2 \ldots a_k\right\} = \frac{1}{b^k}.
$$
It is an easy exercise to show that $\alpha$ is normal in base $b$ iff the sequence $(\{b^n \alpha\})_{n=0,1,\ldots}$ is uniformly distributed in $[0,1)$. That means: for any $a,c$ with $0 \leq a < c \leq 1$ we have
$$
\lim_{N \rightarrow \infty} \frac{1}{N} \# \left\{0 \leq n < N \left|\right. a \leq \left\{b^n \alpha\right\} < c \right\} = c-a.
$$
It was pointed out, for example in \cite{Lar}, that the discrepancy $D_N$ of the sequence $(\{b^n \alpha\})_{n=0,1,\ldots}$ therefore is a perfect measure for the ``quality of the normality of $\alpha $ in base $b$'' (the arguments used in \cite{Lar} for base $b=2$ hold for arbitrary base $b\geq 2$).

Here for an arbitrary sequence $(x_n)_{n=0,1,\ldots}$ in $[0,1)$ the discrepancy $D_N$ is defined by
$$
D_N:=\sup_{0 \leq a < c \leq 1} \left| \left. \frac{1}{N} \# \Big\{0 \leq n < N \right| a \leq x_n < c\Big\} - (c-a)\right|.
$$
We will say: ``$D_N$ is the discrepancy of the normal number $\alpha $ in base $b$''. 
The star-discrepancy $D_N^*$ is obtained by setting $a=0$. Obviously, $D_N^*\leq D_N$ and $D_N\leq 2D^*_N$.

It was shown by W.M. Schmidt \cite{Schm}, and it is a well-known fact that there is a positive constant $c$, such that for every sequence $(x_n)_{n = 0,1, \ldots}$ in $[0,1)$ we have
$$
D_N\geq c \cdot \frac{\log N}{N}
$$

for infinitely many $N$. So also the discrepancy of any normal number $\alpha$ in base $b$ is at least of order $c \cdot \frac{\log N}{N}$. 

By an ingenious construction Mordechay B. Levin in \cite{Lev} provided a number $\lambda$ that is normal in base $2$ with discrepancy $D_N \leq c\cdot \frac{(\log N)^2}{N}$ (with an absolute constant $c>0$). Until then it was only known that for almost all $\alpha$ we have $D_N = \mathcal{O} \left(\left(\frac{\log\log N}{N}\right)^{\frac{1}{2}}\right)$, see \cite{Gal}, and Korobov has given an explicit example of $\alpha$ with $D_N = \mathcal{O} \left(\left(\frac{1}{N}\right)^{\frac{1}{2}}\right)$, see \cite{Koro}. The most prominent normal number, the Champernowne number $\alpha$, is of rather bad quality. We have
$$
D_N \geq c\cdot \frac{1}{\log N}
$$
with $c>0$ for infinitely many $N$. See for example \cite{Schiff}.\\

Nevertheless, there is still a gap of one $\log N$-factor between the best known example of Levin and the currently best-known lower bound for $D_N$. So the  main and certainly challenging question is, if either the upper or the lower bound (or both bounds) for the discrepancy of normal numbers can be improved. The first idea in an attempt to improve the upper bound was to try to improve the discrepancy estimate given by Levin for his normal number $\lambda$. Indeed the discrepancy estimate for Levin's example cannot be improved. This was shown in \cite{Lar}. Namely, we have $N\cdot D_N \geq c \cdot \left(\log N\right)^2$ for infinitely many $N$ (with a fixed positive constant c) for Levin's example.

In \cite{Bec} it was shown, that Levin's normal number in base $2$ can be seen as one special example of a more general concept. This is the concept of $b$-ary $(k,l)$-nested perfect necklaces, together with a technique to concatenate such necklaces to build normal numbers with small discrepancy of order at most $\frac{\left(\log N\right)^2}{N}$.

It is one goal of this paper to show that the lower bound $D_N\geq c \left(\frac{\left(\log N\right)^2}{N} \right)$ with $c>0$ for infinitely many $N$, which was shown in \cite{Lar} for the special example $\lambda$ of Levin in the special base $2$, also holds for many further explicitly known normal numbers generated via perfect necklaces in the manner of Becher and Carton in arbitrary prime base. In Theorem \ref{thm:2} we ensure such a lower bound for Levin's normal number in any prime base $p$. Each additional example with known exact discrepancy order $(\log N)^2/N$ in $N$ supports the hypothesis that such a rate in $N$ is best possible for the discrepancy of any normal number. 

Additionally, we will show (Theorem \ref{thm:1}) that the concept of $b$-ary $(k,l)$-nested perfect necklaces can be weakened to $b$-ary $(k,l)$-nested semi-perfect necklaces, and this concept still can provide normal numbers with a discrepancy of order $\frac{\left(\log N\right)^2}{N}$. An enlargement of the set of normal numbers with good discrepancy allows more space for the search for a normal number with smaller discrepancy than $O(\frac{\left(\log N\right)^2}{N})$. \\ 

The rest of the paper is organized as follows:

In Sections \ref{sec:2} and \ref{sec:3} we reintroduce the concept of $(k,l)$-nested (semi-)perfect necklaces and the corresponding techniques to generate ``low-discrepancy normal numbers'', and we state and prove an upper discrepancy estimate (see Theorem \ref{thm:2} and Example \ref{examp:1}). 
In Section \ref{sec:4} we formulate the lower bound for the discrepancy of normal numbers generated via Levin's construction method in prime base (cf. Theorem \ref{thm:2}) and we give auxiliary results which are needed for the proof. Finally, in Section \ref{sec:5} we give the proof of Theorem \ref{thm:2} and discuss the limits or difficulties of the method of proof for the case of more general normal numbers built by necklaces (see Remark \ref{rem:3}).

\section{Necklaces and normal numbers}\label{sec:2}

In this section we reintroduce and slightly generalize so-called nested perfect necklaces which are used as building blocks of normal numbers in the sense of Becher and Carton (see \cite{Bec}). 

 Let $b\in\NN\setminus\{1\}$ and $A$ be an alphabet with $b$ elements. Here $A=\{0,1,\ldots,b-1\}$. Let $l\in\NN$. A $b$-ary word of length $l$ is a finite sequence of $l$ elements of $A$. E.g. $01110000101$ is a word of length $11$ with $b=2$. Such a word is often called a \textit{binary} word. A necklace or circular word is the equivalence class of a word under rotations. So, e.g., $1100$, $1001$, $0011$, and $0110$ are equivalent.

\begin{definition}
A $b$-ary necklace $w$ is \emph{$(k,l)$-perfect} if its length is $lb^k$ and each word of length $k$ occurs exactly $l$ times at positions which are different modulo $l$. 

A $b$-ary necklace $w$ we call \emph{$(k,l)$-semi perfect} if its length is $lb^k$ and each word of length $k$ occurs exactly $l$ times. 

A $b$-ary necklace $w$ is \emph{$(k,l)$-nested perfect} if for each integer $j\in\{1,\ldots,k\}$ each block of $w$ of length $lb^j$ starting at a position congruent to 1 modulo $lb^j$ is a $(j,l)$-perfect necklace. 

We call a necklace \emph{$(k,l)$-nested semi-perfect} if for each integer $j\in\{1,\ldots,k\}$ each block of $w$ of length $lb^j$ starting at a position congruent to 1 modulo $lb^j$ is a $(j,l)$-semi perfect necklace. 
 
\end{definition}

The following theorem gives a discrepancy bound for the sequences $(\{b^n\alpha\})_{n\geq 0}$, where the $b$-ary expansion of $\alpha$ is obtained by the concatenation of $b$-ary nested semi-perfect necklaces. 
\begin{theorem}\label{thm:1}
Let $b\in\NN\setminus\{1\}$, $g,f:\NN\to\NN$, and $\alpha\in[0,1)$ such that its base $b$ representation is obtained by the concatenation of $b$-ary $(f(j),g(j))$-nested semi-perfect necklaces, with $j=1,2,3,\ldots$. 

Then the star discrepancy $D_N^*$ of $(\{b^n\alpha\})_{n=0,1,\ldots}$ satisfies 
$$ND^*_N\leq \sum_{j=1}^{m-1}f(j)+\sum_{j=1}^{m}g(j)+(b-1)\frac{f(m)^2}{2}+(b-1)f(m)g(m)$$ for all $N\in\NN$, where $m\in\NN$ is such that $\sum_{j=1}^{m-1}g(j)b^{f(j)}\leq N< \sum_{j=1}^{m}g(j)b^{f(j)}$.
\end{theorem}
\begin{proof}
We write $N$ as $N=\sum_{j=1}^{m-1}g(j)b^{f(j)}+N'=:n_{m}+N'$, $N'$ as
$N'=g(m)N''+R$ with $0\leq R< g(m)$, and $N''$ as $N''=N_0+N_1b^1+N_2b^2+\cdots+N_{f(m)-1}b^{f(m)-1}$, i.e., its unique base $b$ representation with $N_i\in\{0,1,\ldots,b-1\}$.

Let $[a,b)\subseteq [0,1)$ and $A(\mathcal{P},[a,b)):=\#\{x_n\in \mathcal{P}:x_n\in[a,b)\}$ be a counting function and $\Delta(\mathcal{P},[a,b)):=A(\mathcal{P},[a,b))-|\mathcal{P}|\cdot(b-a)$ be the so-called local discrepancy function where $\mathcal{P}$ might be a multiset. 

Let $\gamma\in[0,1)$ and $j\in\{1,\ldots,m-1\}$. We estimate $A(\mathcal{P}_j,[0,\gamma))$ for $\mathcal{P}_j=\{\,\{b^n\alpha\}\,:\,n_j\leq n< n_{j+1} \}$. We define $c:=\lfloor\gamma b^{f(j)}\rfloor\in[0,b^{f(j)})\cap \NN_0$. Note that 
$$ A(\mathcal{P}_j,[0,c/b^{f(j))})\leq A(\mathcal{P}_j,[0,\gamma))\leq A(\mathcal{P}_j,[0,(c+1)/b^{f(j)})).$$
For the outer terms we use the necklace property and obtain 
$$\Delta(\mathcal{P}_j,[0,c/b^{f(j)}))=A(\mathcal{P}_j,[0,c/b^{f(j)}))-g(j)c= \pm \rho_j,\,\mbox{ with } 0\leq \rho_j\leq f(j)-1.$$
Note that each $u$ with $0\leq u\leq c-1$ corresponds with a block of length $f(j)$ and each such block occurs exactly $g(j)$ times in a $(f(j),g(j))$-semi-perfect necklace. Note that at most the $f(j)-1$ blocks that are crossing the necklace clasp might be different from the corresponding digit-blocks in $\alpha$. \\

Therefore, it is easy to check that the local discrepancy function satisfies 
$$|\Delta(\mathcal{P}_j,[0,\gamma))|\leq f(j)+g(j)-1.$$
Note that 
\begin{eqnarray*}
\Delta(\mathcal{P}_j,[0,c/b^{f(j)}))-|\mathcal{P}_j|1/b^{f(j)}\leq \Delta(\mathcal{P}_j,[0,\gamma))\leq \Delta(\mathcal{P}_j,[0,(c+1)/b^{f(j)}))+|\mathcal{P}_j|1/b^{f(j)}.
\end{eqnarray*}

Using the base $b$ representation of $N''=N_0+N_1b^1+N_2b^2+\cdots+N_{f(m)-1}b^{f(m)-1}$, we define $B_{l,k}:=kb^l+\cdots+N_{f(m)-1}b^{f(m)-1}$ with $k\in\{0,1,\ldots,N_l\} $ for those $l\in\{0,1,\ldots,f(m)-1\}$ such that $N_l\neq 0$. Furthermore for such $k,l$, we set $\mathcal{B}_{l,k}:=\{\,\{b^n\alpha\}\,:\,n_m+B_{l,k}\leq  n< n_m+B_{l,k+1}\}$ for those $l\in\{0,1,\ldots,f(m)-1\}$ such that $N_l\neq 0$ with $k\in\{0,1,\ldots,N_l-1\}$. 

By the nested property of the necklaces and analog argumentation as above we obtain 
$$|\Delta(\mathcal{B}_{l,k},\gamma)|\leq l+g(m)-1.$$

Using the additive property in the first argument of the local discrepancy we obtain 
\begin{eqnarray*}
ND_N^*&\leq &\sum_{j=1}^{m-1}(f(j)+g(j)-1)+\sum_{l=0}^{f(m)-1}\sum_{k=0}^{N_l-1}(l+g(m)-1)+g(m)-1\\
&\leq &\sum_{j=1}^{m-1}f(j)+\sum_{j=1}^{m}g(j)\\
&&+(b-1)\frac{f(m)(f(m)-1)}{2}+(b-1)f(m)g(m)-(b-1)f(m)-m.
\end{eqnarray*}
\end{proof}
From the proof of Theorem \ref{thm:1} we immediately derive the following corollary, which will be used via Item (1) of Example \ref{examp:1} in the proof of Theorem \ref{thm:2}. 

\begin{corollary}\label{coro:1}
Let $b\in\NN\setminus\{1\}$, $g,f:\NN\to\NN$, and $\alpha\in[0,1)$ such that its base $b$ representation is obtained by the concatenation of $b$-ary $(f(j),g(j))$-nested semi-perfect necklaces, with $j=1,2,3,\ldots$. 
We set $n_m:=\sum_{j=1}^{m-1}g(j)b^{f(j)}$ with $m\geq 2$.
Then the star discrepancy $D_{n_m}^*$ of $(\{b^n\alpha\})_{n=0,1,\ldots,n_m-1}$ satisfies $$n_mD^*_{n_m}\leq \sum_{j=1}^{m-1}f(j)+\sum_{j=1}^{m-1}g(j).$$

\end{corollary}

\begin{example}\label{examp:1}
\begin{enumerate}
	\item Let $g(j)=f(j)=b^j$. Then $m\leq \log_b\log_b N+1$ for $N$ large enough and the star discrepancy $D_N^*$ of $(\{b^n\alpha\})_{n=0,1,\ldots}$ satisfies
	$$ND_N^*=O_b(\log^2N).$$
	Furthermore, for $n_m:=\sum_{j=1}^{m-1}b^jb^{b^j}$ where $m\geq 2$ we have
	\begin{equation*}
	n_mD^*_{n_m}=O_b(\log N).
	\end{equation*}
	\item Let $g(j)=f(j)=j$. Then $m\leq \log_b N+1$ for $N$ large enough and the star discrepancy $D_N^*$ of $(\{b^n\alpha\})_{n=0,1,\ldots}$ satisfies 
	$$ND_N^*(\{b^n\alpha\})=O_b(\log^2N).$$
	
\end{enumerate}
\end{example}

\begin{remark}\label{rem:kl}{\rm
Interesting questions on necklaces include, e.g., which values of $k$ and $l$ are possible such that a $(k,l)$-semi perfect nested $b$-ary necklace can exist. For example, it is not so hard to see, that $$l\geq \frac{k}{2}.$$
From the semi-perfect property we know that the word of length $l b^k$, i.e., 
$$\underbrace{\overbrace{\ldots}^{b^1 l}\overbrace{\ldots}^{b^1 l} \quad \quad\ldots\ldots\ldots\quad \quad \overbrace{\ldots}^{b^1 l}}_{\mbox{$lb^k$ elements in $\{0,1,\ldots,b-1\}$}},$$ must contain a block consisting of $k$ consecutive $0$s.

From the nested property we know that each subword of length $lb^1$ starting at a position congruent to $1$ modulo $lb^1$ must contain $l$ $0$s, $l$ $1$s, $\ldots$, and $l$ $(b-1)$s. Hence the block consisting of $k$ consecutive $0$s, might be part of at most two such neighboring subwords of length $b^1l$. Hence 
$$\overbrace{\ldots\ldots\ldots\ldots\underbrace{00\ldots\ldots0}_{\mbox{at most  $l\, 0$s, }}}^{b^1 l}  \overbrace{\underbrace{00\ldots\ldots0}_{\mbox{ at most $ l\, 0$s}}\ldots\ldots\ldots\ldots}^{b^1 l}.$$
Therefore, $2l\geq k$ or equivalently $l\geq \frac{k}{2}$. 
}
\end{remark}

At this point we also note the following lemma on normal numbers built of necklaces, which will be needed for the proof of Theorem \ref{thm:2}. 

\begin{lemma} \label{lem:2gen}
Let $b\in\NN\setminus\{1\}$ and $\alpha$ be a number in $[0,1)$ whose base $b$ representation is built by the concatenation of $(b^j,b^j)$-semi perfect nested $b$-ary necklaces with $j=1,2,3,\ldots$. 

For every positive integer $m$ we define $n_m:=\sum_{j=1}^{m-1}b^jb^{b^j}$. Then for every $1 \leq t \leq b^m$, every integer $B$ with $0 \leq B < b^{b^{m}-t}$, and for every integer $c$ with $0 \leq c < b^{t}$ with the exception of at most $2b^{m}$ such $c$ we have
$$
\# \underbrace{\left\{(k,n)\,|\,0 \leq k < b^m, B \cdot b^{t} \leq n < B \cdot b^{t} + b^{t} \left |\left\{b^{n_m + {b^m} n+k} \cdot \alpha\right\}\right.\in \left[\frac{c}{b^{t}}, \frac{c+1}{b^{t}}\right)\right\} }_{=:\mathcal{M}_c}= b^m.
$$
\end{lemma}
\begin{proof}
From the definition of a semi-perfect nested necklace, which we notate by $$\eta_1,\,\eta_2,\,\ldots,\,\eta_{b^{m+t}},$$ we know that each block of length $t$ occurs exactly $b^m$ times. Only the $t-1$ blocks at the necklace clasp (at the beginning and at its end), they are $$(\eta_{b^{m+t}-(t-1)},\eta_{b^{m+t}-(t-2)},\ldots, \eta_1),\,\ldots,\,(\eta_{b^{m+t}},\eta_{1},\ldots, \eta_{t-1}), $$ might be different from the corresponding $(t-1)$ blocks in the base $b$ representation of $\alpha$, which we notate by 
\begin{align*}
\ldots,\alpha_{n_m+b^{m+t}B},\alpha_{n_m+b^{m+t}B+1},\ldots\ldots,\alpha_{n_m+b^{m+t}B+b^{m+t}-1},\\
\alpha_{n_m+b^{m+t}(B+1)},\alpha_{n_m+b^{m+t}(B+1)+1},\ldots,\alpha_{n_m+b^{m+t}(B+1)+t-2},\ldots.
\end{align*} 
Thus for at most $2(t-1)< 2b^m$ values of $c$ the set $\mathcal{M}_c$ might not get the fair proportion of $b^m$ pairs $(k,n)$.
\end{proof}

\section{Constructing $\alpha$ by concatenating affine necklaces over $\FF_p$}\label{sec:3}

In this section we reintroduce so-called affine necklaces, in the sense of Becher and Carton \cite{Bec} for base $2$, in arbitrary prime base $p$.

Let $p\in\PP$. We construct a normal number $\alpha$ in base $p$ (cf. \cite{Bec}) as follows: We denote the representation of $\alpha$ in base $p$ by\\

$\alpha = 0 . \underset{\mathcal{A}_1}{\underbrace{\alpha_1\alpha_2 \ldots \alpha_{p^{(p^1+1)}}}} ~ \underset{\mathcal{A}_2}{\underbrace{\alpha_{p^{(p^1+1)}+1} \ldots \alpha_{p^{(p^1+1)}+p^{(p^2+2)}}}} \quad \ldots\ldots \quad \underset{\mathcal{A}_m}{\underbrace{\ldots\ldots\ldots\ldots}} \quad \ldots\ldots$.\\

Here the block $\mathcal{A}_m$ for $m\in\NN$ consists of $p^m \cdot p^{p^{m}}$ digits $\alpha_i$, which are constructed as follows. 

Let $m\in\NN$. Set $n_m := \sum_{j=1}^{m-1}p^j \cdot p^{p^{j}}$. Then block $\mathcal{A}_m$ starts with $\alpha_{n_m +1}$. We write the block $\mathcal{A}_m$ in the form

$$
\underbrace{d_0 (0) \ldots d_k(0) \ldots d_{p^m-1}(0)} \quad \underbrace{d_0(1) \ldots d_k(1) \ldots d_{p^m-1}(1)} \quad \ldots
$$
$$
\ldots \quad \underbrace{d_0(n) \ldots d_k(n) \ldots d_{p^m-1}(n)}  \quad\ldots \quad \underbrace{d_0(p^{p^{m}}-1)  \ldots d_{k}(p^{p^{m}}-1)  \ldots  d_{p^m-1}(p^{p^{m}}-1)}.
$$
The $d_k(n)$ for $k\in\{0,1,\ldots,p^m-1\}$ and for $n$ between $0$ and $p^{p^{m}}-1$ are computed as follows. 

Let $\boldsymbol{z}=\boldsymbol{z}(m)=(z_j)_{0\leq j<p^m}\in\{0,1,\ldots,p-1\}^{p^m}$ , $\boldsymbol{\eta}=\boldsymbol{\eta}(m)=(\eta_j)_{0\leq j<p^m}\in(\NN_0)^{p^m}$ and $\boldsymbol{u}=\boldsymbol{u}(m)=(u_j)_{0\leq j<p^m}\in\{0,1,\ldots,p-1\}^{p^m}$ be such that $\eta_0=0$ and $\eta_{j}\leq \eta_{j+1}\leq \eta_{j}+1$ for all $j\in\{0,1,\ldots,p^m-2\}$ and $u_j\not\equiv 0\pmod{p}$ for all $j\in\{0,1,\ldots,p^m-1\}$. 

We start with the unique base $p$ representation of $n\in\{0,1,\ldots., p^{p^m}-1\}$, i.e., $$n= e_0(n) + p \cdot e_1(n) + \ldots + p^{p^{m}-1} \cdot e_{p^m-1} (n)$$ with $e_i\in\{0,1,\ldots,p-1\}$ for $i=0,1,\ldots,{p^m}-1$.\\

Then we set $d_k(n) := p^{(\boldsymbol{\eta},\boldsymbol{u})}_{k,0} e_0(n) + \ldots + p^{(\boldsymbol{\eta},\boldsymbol{u})}_{k,p^m-1}e_{p^m-1}(n) +z_k\mod p$, where 
$$p^{(\boldsymbol{\eta},\boldsymbol{u})}_{i,j} := \binom{i+j-\eta_j}{j}u_j$$ for all non-negative integers $i,j\in\{0,1,\ldots,p^m-1\}$. \\

Such a block $\mathcal{A}_m$ is called \textit{an affine necklace} (cf. Becher and Carton \cite{Bec} for the special case where $p=2$ and $u_j=1$). \\

\begin{example}\label{examp:2}
Levin's normal number in arbitrary prime base $p$ is obtained by the construction described above with the setting $\boldsymbol{z}=\boldsymbol{0}$, $\boldsymbol{\eta}=\boldsymbol{0}$ and $\boldsymbol{u}=\boldsymbol{1}$ for every $m\in\NN$. (Here and later on, for $m\in\NN$ we set $\boldsymbol{0}:=(0)_{0\leq j<p^m}$ and $\boldsymbol{1}:=(1)_{0\leq j<p^m}$). 
\end{example}

In the following we identify specific properties of the coefficients $p^{(\boldsymbol{\eta},\boldsymbol{u})}_{i,j}$. Therefore, we define for $m\in\NN$ the $p^m \times p^m$-matrix $P^{(\boldsymbol{\eta},\boldsymbol{u})}$ as
\begin{eqnarray*}
P^{(\boldsymbol{\eta},\boldsymbol{u})}&:=& \left(p^{(\boldsymbol{\eta},\boldsymbol{u})}_{i,j}\right)_{0\leq i, j<p^m} = 
\begin{pmatrix} 
\binom{0-\eta_0}{0}u_0 & \binom{1-\eta_1}{1}u_1 & \ldots & \binom{{p^m-1}-\eta_{p^m-1}}{p^m-1}u_{p^m-1} \\
\binom{1-\eta_0}{0}u_0 & \binom{1+1-\eta_1}{1}u_1 & \ldots & \binom{{p^m-1}+1-\eta_{p^m-1}}{p^m-1}u_{p^m-1} \\
\vdots  & \ldots &\ldots  & \vdots \\
\binom{p^m-1-\eta_0}{0}u_0 & \binom{p^m-1+1-\eta_1}{1}u_1 & \ldots & \binom{{p^m-1}+p^m-1-\eta_{p^m-1}}{p^m-1}u_{p^m-1} \\ \end{pmatrix}\\
&=&\left(\binom{i+j-\eta_j}{j}u_j\right)_{0\leq i,j<p^m}.
\end{eqnarray*}

\begin{figure}[h]
	\centering
		\includegraphics[width=0.30\textwidth]{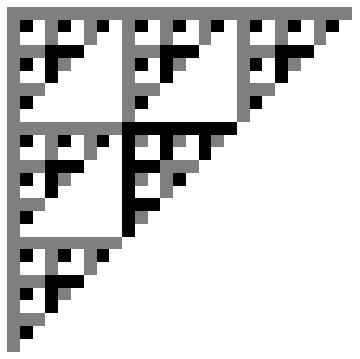}
	\hspace{0.7cm}\includegraphics[width=0.30\textwidth]{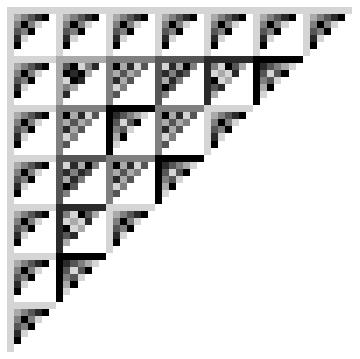}
	\caption{Upper left $p^m\times p^m$ submatrices of $P^{(\boldsymbol{\eta},\boldsymbol{u})}\pmod{p}$ with $\boldsymbol{u}=\boldsymbol{1}$ and $\boldsymbol{\eta}=\boldsymbol{0}$ for $p=3,\,m=3$ (left) and $p=7,\,m=2$ (right).}\label{fig:Mat_1}
\end{figure}

\begin{figure}[h]
	\centering
		\includegraphics[width=0.30\textwidth]{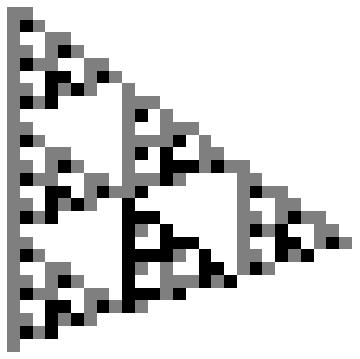}
	\hspace{0.7cm}\includegraphics[width=0.30\textwidth]{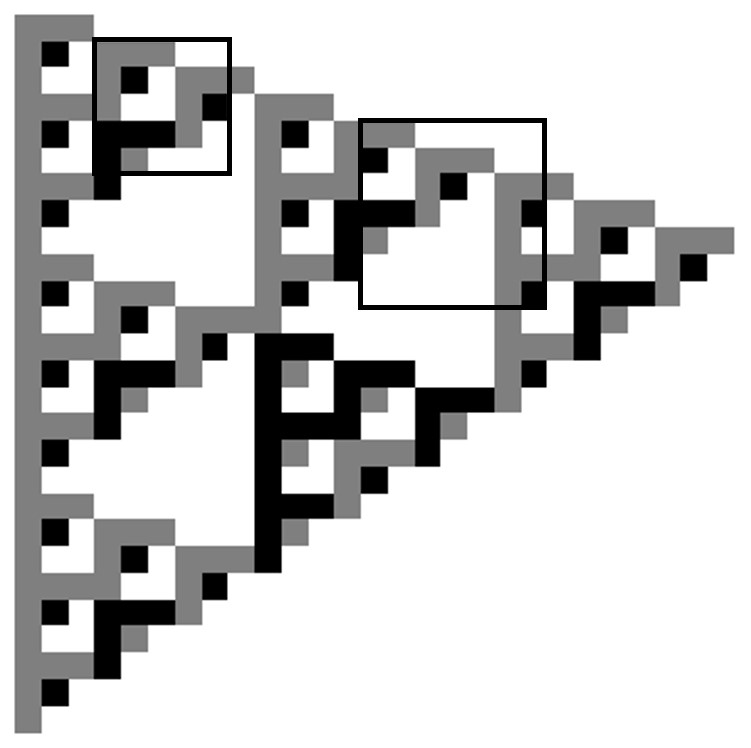}
	\caption{Upper left $3^3\times 3^3$ submatrices of $P^{(\boldsymbol{\eta},\boldsymbol{u})}\pmod{3}$ for $\boldsymbol{u}=\boldsymbol{1}$ and with $\eta_j=\lfloor j/\sqrt{2}\rfloor$ (left) and  $\eta_j=\lfloor j/3\rfloor$ (right). The sqares in the right image show two examples of submatrices described by \eqref{P:cond:2} in Lemma \ref{lem:rank} with $t=5$, $l=3$ and therefore $\eta_l=1$ and with $t=7, l=13$ and therefore $\eta_l=4$.}\label{fig:Mat_2}
\end{figure}

We collect some nice properties of these $p^m \times p^m$ matrices $P^{(\boldsymbol{\eta},\boldsymbol{u})}$ modulo $p$ with $p\in\PP$. Figures \ref{fig:Mat_1} and \ref{fig:Mat_2} show some examples of such matrices. In particular, Figure \ref{fig:Mat_2} shows that $\eta_j$ denotes a downshift of the $j$th column of the ``unshifted'' case with $\boldsymbol{\eta}=\boldsymbol{0}$, see Figure \ref{fig:Mat_1}. 

\begin{lemma}\label{lem:rank} 
Let $p\in\PP$ and $m\in\NN$. Then the $p^m\times p^m$ matrix $P^{(\boldsymbol{\eta},\boldsymbol{u})}\pmod{p}$ 
\begin{enumerate}
	\item in the case where $\boldsymbol{\eta}=\boldsymbol{0}$ and $\boldsymbol{u}=\boldsymbol{1}$ is an upper anti-diagonal triangular matrix with $1$s in the first row and first column, with an anti-diagonal containing just nonzero entries. 
	\item satisfies that any $t\times t$ square submatrix with $t\in\{1,\ldots,p^m\}$ of the form 
	\begin{equation}\label{P:cond:1}
	\left(\binom{i+j-\eta_j}{j}u_j\right)_{i=l,\ldots,l+t-1,\,j=0,\ldots t-1},\end{equation}
	with $l\in\NN_0$ such that $l+t\leq p^m$ is regular modulo $p$.
	\item satisfies that any $t\times t$ square submatrix with $t\in\{1,\ldots,p^m\}$
	\begin{equation}\label{P:cond:2}
	\left(\binom{i+j-\eta_j}{j}u_j\right)_{i=\eta_l,\ldots,\eta_l+t-1,\,j=l,\ldots l+t-1},\end{equation}
	with $l\in\NN_0$ such that $l+t\leq p^m$, 
i.e. the upper left corner is at a position with an entry stemming from the first row of the unshifted case $\boldsymbol{\eta}=\boldsymbol{0}$, is regular modulo $p$. 
\end{enumerate}
\end{lemma}

\begin{proof}
The statement in the first item on the first column and on the first row is obvious. The claim for the anti-diagonal entries, (i.e. $i+j=p^m-1)$, follows from Lucas' Theorem (see Lemma \ref{lem:Lucas_Thm}) $\binom{i+j}{j}=\binom{p^m-1}{j}\equiv\prod_{k=0}^{m-1}\binom{p-1}{j_k}\not\equiv 0\pmod{p}$ where $j=j_0+j_1p+\cdots+j_{m-1}p^{m-1}$ is the unique base $p$ representation of $j$. For the lower anti-diagonal entries note that if $i+j\geq p^m$ then (again by Lucas' Theorem see Lemma \ref{lem:Lucas_Thm}) $\binom{i+j}{j}\equiv\binom{1}{0}\binom{i+j-p^m}{j}\equiv 0 \pmod{p}$ as obviously $i+j-p^m<j$. \\

The second item in the case where $\boldsymbol{\eta}=\boldsymbol{0}$ is an immediate consequence of \cite[Theorem~1.3]{Bic} which ensures that the corresponding submatrices of Pascal's matrix $(\binom{i+j}{j})_{i,j\geq 0}$ have nonzero determinant. 

For the claim of first $t\times t$ submatrix \eqref{P:cond:1}, we proceed as in \cite[Proof of Lemma~4]{Bec} and use $\binom{i+j-\eta_j}{j}-\binom{i+j-\eta_j-1}{j}=\binom{i+j-\eta_j-1}{j-1}$ for proper row-manipulations to end up with an upper right regular triangular matrix. 
We leave the first row unchanged and substitute the $k$th row for $k=2,\ldots,t$ with the $k$th row minus the $(k-1)$th row. Consequently, the first column of the new matrix is $(u_0,0,\ldots,0)^T$ and its lower right $(t-1)\times (t-1)$-submatrix $Q_{t-1}$ satisfies
$$Q_{t-1}=\left(\left\{\begin{array}{ll}
	\binom{i+j-1-\eta_{j-1}}{j-1}u_j&\mbox{ if }\eta_j=\eta_{j-1}\\
	\binom{i-1+j-1-\eta_{j-1}}{j-1}u_j&\mbox{ if }\eta_j=\eta_{j-1}+1
\end{array}\right.\right)_{1\leq i\leq t-1,1\leq j\leq t-1}.$$
Note that the first column of $Q_{t-1}$ consists of $u_1$ entries exclusively. 

In the second step we leave the first row in $Q_{t-1}$ unchanged and substitute the $k$th row for $k=2,\ldots,t-1$ with the $k$th row minus the $(k-1)$th row. Then 
the first column of the new matrix is $(u_1,0,\ldots,0)^T$ and its lower right $(t-2)\times (t-2)$-submatrix $Q_{t-2}$ satisfies
$$Q_{t-2}=\left(\left\{\begin{array}{ll}
	\binom{i+j-2-\eta_{j-2}}{j-2}u_j&\mbox{ if }\eta_j=\eta_{j-2}\\
	\binom{i-1+j-2-\eta_{j-2}}{j-2}u_j&\mbox{ if }\eta_j=\eta_{j-2}+1\\
	\binom{i-2+j-2-\eta_{j-2}}{j-2}u_j&\mbox{ if }\eta_j=\eta_{j-2}+2
\end{array}\right.\right)_{2\leq i\leq t-1,2\leq j\leq t-1}.$$
We repeat this step and obtain an upper triangular matrix with diagonal $(u_0,u_1,\ldots,u_{t-1})$ that is of course regular. \\

For the claim of the second $t\times t$ submatrix \eqref{P:cond:2}, we proceed as in \cite[Proof of Lemma~5]{Bec} by analogous column manipulations as for the submatrix above. First we multiply each column with $u_j^{-1}$. Note that the first column of this submatrix is of the form $(1,\binom{l+1}{l},\ldots,\binom{l+t-1}{l})^T$. We leave the first column unchanged and substitute the $k$th column, for $k=l+1,\ldots,l+r$ where $r\in\{0,1,\ldots,t-1\}$ is chosen maximal such that $\eta_l=\eta_{l+1}=\cdots=\eta_{l+r}$, with the $k$th column minus $(k-1)$th column. Consequently, the first row of the new matrix is $(1,0,\ldots,0)$ and its lower right $(t-1)\times (t-1)$-submatrix $Q_{t-1}$ satisfies
$$Q_{t-1}=\left(\left\{\begin{array}{ll}
	\binom{i+k-1-\eta_{k}}{k}&\mbox{ if }k\in\{l+1,\ldots,l+r\}\\
	\binom{i+k-\eta_{k}}{k}&\mbox{ else. }
\end{array}\right.\right)_{\eta_{l}+1\leq i\leq \eta_l+t-1,l+1\leq k\leq l+t-1}.$$
This step can be seen as the transformation $\eta_k\mapsto \eta_k+1$ for $k=l+1,\ldots,l+r$. 

For $Q_{t-1}$ we proceed analogously. Then for $Q_{t-2}$ etc. until we end up in the case where $\eta_{l+j}=\eta_l+j$ for $j=1,\ldots t-1$. 
For this setting  $\eta_{l+j}=\eta_l+j$ the $t\times t$ submatrix in \eqref{P:cond:2} is a lower triangular matrix with all diagonal entries equal to $1$. Hence this lower diagonal matrix is of course regular.

\end{proof}

\begin{remark}\label{rem:affine}{\rm 
From Lemma \ref{lem:rank} Item (1) we know that $P^{(\boldsymbol{\eta},\boldsymbol{u})}$ modulo $p$ is invertible modulo $p$. Thus an alternative construction of $d_k(n)$ in the affine necklace $\mathcal{A}_m$ is 
$$d_k(n) := p^{(\boldsymbol{\eta},\boldsymbol{u})}_{k,0} (e_0(n)+z'_0) + \ldots + p^{(\boldsymbol{\eta},\boldsymbol{u})}_{k,p^m-1}(e_{p^m-1}(n)+z'_{p^m-1})\mod p$$
with $\boldsymbol{z}'=(z'_0,\ldots,z'_{p^m-1})$ satifying $P^{(\boldsymbol{\eta},\boldsymbol{u})}(\boldsymbol{z}')^T=(\boldsymbol{z})^T\pmod{p}$. 
}
\end{remark}

Becher and Carton \cite[Lemma~6]{Bec} showed an interdependence between binary nested perfect necklaces, that also holds for $b$-ary necklaces over an alphabet, e.g., $\{0,1,\ldots,b-1\}$ where we use arithmetics modulo $b$. The proof of the following lemma is a straightforward adaption of the arguments of \cite[Proof of Lemma~6]{Bec} that relies on the ``perfect'' property. 

\begin{lemma}\label{lem:affine}
Let $w$ be a word of length $lb^k$ and $z$ be a word of length $l$ over $\{0,1,\ldots,b-1\}$. Furthermore, let $z^{b^k}$ denote the word of length $lb^k$ that is built by concatinating $b^k$ copies of $z$. We define the $b$-ary word $w'=w+z^{b^k}\mod b$. Then $w$ is a $(k,l)$-nested perfect $b$-ary necklace if and only if $w'$ is a $(k,l)$-nested perfect $b$-ary necklace. 
\end{lemma}

\begin{proposition}\label{prop:1}
The affine necklace $\mathcal{A}_m$ is a $(p^m,p^m)$-nested perfect $p$-ary necklace. 
\end{proposition}
\begin{proof}
The proof is carried out in accordance with \cite[Proof of Proposition~7]{Bec}. From the definition of affine necklaces together with Lemma \ref{lem:affine} it is easily seen that it suffices to prove the case where $z'=z=(0,\ldots,0)$. 

Let $w\in\{1,\ldots,p^m\}$, $v:=p^m-w$ and $B=Ap^w$ with $A\in\{0,1,\ldots,p^{v}-1\}$. We consider the subword of $\mathcal{A}_m$ built by concatenating $d_0(n)d_1(n)\ldots d_{p^m-1}(n)$ for $n=Ap^w,Ap^w+1,\ldots,Ap^w+p^w-1$ and have to show that this word of length $p^mp^w$ is a perfect $(w,p^m)$ necklace. 

Let $r\in\{0,\ldots,p^m-1\}$ be the position where the block $(c_0,\ldots,c_{w-1})$ of length $w$ is starting. 

We distinguish two cases: $r+w\leq p^m$ and $r+w> p^m$. 
\begin{itemize}
	\item Let $r+w\leq p^m$. We ask for the number of $n\in\{Ap^w,Ap^w+1,\ldots,Ap^w+p^w-1\}$ such that \begin{equation}\label{cond:case1}d_r(n)d_{r+1}(n)\ldots d_{r+w-1}(n)=c_0\ldots c_{w-1}.\end{equation} We have to ensure that exactly one value of such $n$ satisfies our condition \eqref{cond:case1}. 
	We translate the condition \eqref{cond:case1} to the construction of the corresponding subword of $\mathcal{A}_m$. Equivalently we need to show that
\begin{eqnarray} \label{mat:case1}
\left(\binom{i+j-\eta_j}{j}u_j\right)_{r\leq i<r+w,0\leq j<p^m}\begin{pmatrix}k_0\\\vdots\\k_{w-1}\\a_0\\\vdots\\a_{p^v-1}\end{pmatrix}=\begin{pmatrix}c_{0}\\\vdots\\c_{w-1}, \end{pmatrix}
\end{eqnarray}
where $a_0+a_1p+\cdots+a_{p^v-1}p^{p^v-1}$ is the unique $p$-ary expansion of $A$ and $k_0+k_1p+\cdots +k_{w-1}p^{w-1}$ is the unique $p$-ary expansion of $k\in\{0,1,\ldots,p^w-1\}$, has a unique solution for $k\in\{0,1,\ldots,p^w-1\}$. 

From Lemma \ref{lem:rank} Item (2) we know that the left square $w\times w$ submatrix of the matrix in \eqref{mat:case1} is regular and the result immediately follows. 

\item Let $r+w> p^m$. We define $r_1:=r+w-p^m$ and $r_2=w-r_1$ and ask for the number of $n\in\{Ap^w,Ap^w+1,\ldots,Ap^w+p^w-1\}$ such that \begin{equation}\label{cond:case2}d_{p^m-r_1}(n)\ldots d_{p^m-1}(n)d_0(n+1)\ldots d_{r_2-1}(n+1)=c_0\ldots c_{w-1}.\end{equation}
	We translate the condition \eqref{cond:case2} to the construction of the corresponding subword of $\mathcal{A}_m$: We need to prove that 
\begin{eqnarray*} 
\left(\binom{i+j-\eta_j}{j}u_j\right)_{p^m-r_1\leq i<p^m,0\leq j<p^m}\begin{pmatrix}k_0\\\vdots\\k_{w-1}\\a_0\\\vdots\\a_{p^v-1}\end{pmatrix}=\begin{pmatrix}c_{0}\\\vdots\\c_{r_1-1}, \end{pmatrix}
\end{eqnarray*}

together with 

\begin{eqnarray*}
\left(\binom{i+j-\eta_j}{j}u_j\right)_{0\leq i<r_2,0\leq j<p^m}\begin{pmatrix}(k+1)_0\\\vdots\\(k+1)_{w-1}\\a_0\\\vdots\\a_{p^v-1}\end{pmatrix}=\begin{pmatrix}c_{r_1}\\\vdots\\c_{w-1}, \end{pmatrix}.
\end{eqnarray*}

have a unique solution $k\in\{0,1,\ldots,p^w-1\}$, where $k+1$ for $k=p^w-1$ is set $0$.

Or equivalently 
\begin{eqnarray} \label{mat:case21}
\left(\binom{i+j-\eta_j}{j}u_j\right)_{p^m-r_1\leq i<p^m,0\leq j<w}\begin{pmatrix}k_0\\\vdots\\k_{w-1}\end{pmatrix}=\begin{pmatrix}c'_{0}\\\vdots\\c'_{r_1-1}, \end{pmatrix}
\end{eqnarray}

and 

\begin{eqnarray}\label{mat:case22}
\left(\binom{i+j-\eta_j}{j}u_j\right)_{0\leq i<r_2,0\leq j<w}\begin{pmatrix}(k+1)_0\\\vdots\\(k+1)_{w-1}\end{pmatrix}=\begin{pmatrix}c'_{r_1}\\\vdots\\c'_{w-1}, \end{pmatrix}.
\end{eqnarray}

have a unique solution $k\in\{0,1,\ldots,p^w-1\}$, where again $k+1$ for $k=p^w-1$ is set $0$.

We search the rightmost $r_2\times r_2$-submatrix $$\left(\binom{i+j-\eta_j}{j}u_j\right)_{0\leq i<r_2,s\leq j<s+r_1}$$ in \eqref{mat:case22} such that in the upper left corner is a nonzero entry. Hence $\eta_s=\eta_{s-1}=\cdots=\eta_0=0$. 

First, suppose $s=r_2$, then the left $r_1\times r_1$ submatrix in \eqref{mat:case21} is a regular upper anti-diagonal triangular matrix and the right $r_1\times r_2$ submatrix is the zero matrix. Hence the first $r_1$ digits of $k$ and therefore of $k+1$ are uniquely determined. 
From the third item in Lemma \ref{lem:rank} we know that then the $r_2\times r_2$-submatrix $$\left(\binom{i+j-\eta_j}{j}u_j\right)_{0\leq i<r_2,r_2\leq j<r_1+r_2}$$
is regular and the last $r_2$ digits of $k+1$ are uniquely determined. 

The only problem occurs in the case where the first $r_2$ digits of $k+1$ have to be $0$ and the last $r_1$ digits of $k$ have to be $q-1$, then the solution of $k$ must be $p^w-1$ and the block occurs indeed at the necklace clasp. 

In the general case $s\leq r_2$ we observe that from the definition of $P^{(\boldsymbol{\eta},\boldsymbol{u})}$ the rows in the matrices in \eqref{mat:case22} and \eqref{mat:case21} might be rearranged in a regular lower-triangular matrix. 
The first row determines the first digit of $k$ or $k+1$ respectively, depending on which submatrix the first row comes from. The second row determines then the second digit of $k$ or $k+1$ respectively, etc. 
Finally, $r_1$ digits of $k$ are fixed and the remaining $r_2$ digits of $k+1$ are fixed. We have to investigate which combination of digits might have more than one solution or no solution for $k$.  
No problem occurs if the first digit of $k$ is $\neq q-1$ or the first digit of $k+1$ is $\neq 0$. Otherwise we have to check if the carries from $k$ to $k+1$ might be possible. Again the only problem might occur if the fixed digits of $k$ are all $q-1$ and the fixed digits of $k+1$ are $0$. But again we have exactly one solution for such $k$ at the necklace clasp.

\end{itemize}

\end{proof}

From Proposition \ref{prop:1}, Theorem \ref{thm:1}, and the first item in Example \ref{examp:1} we immediately obtain the following Corollary \ref{coro:2}. 

\begin{corollary}\label{coro:2}
Let $\alpha\in[0,1)$ be such that its $p$-ary representation is the concatenation of affine necklaces $\mathcal{A}_m$ for $m=1,2,3,\ldots$. Then its star discrepancy $D_N^*$ satisfies
$$ND^*_N=O_p(\log^2 N).$$ 
\end{corollary}

\begin{remark}
In \cite{Bec} it was shown that the set of all affine necklaces with $m\in\NN$ over $\FF_2$ equals the set of all binary $(2^m,2^m)$-nested perfect necklaces. It is an open problem to compare the set of all $p$-ary $(p^m,p^m)$-nested (semi-)perfect necklaces with the set of all affine necklaces with $m\in\NN$ over $\FF_p$.  
\end{remark}

\section{The lower discrepancy bound. Auxiliary results}\label{sec:4}

In \cite{Lar} the authors showed that the star-discrepancy $D_N^*$ of Levin's \cite{Lev} binary normal number $\lambda$ satisfies $ND^*_N\geq c(\log^2 N)$ with $c>0$ for infinitely many $N$. Levin's $\lambda$ represents the most basic example of a normal number built by the concatenation of affine necklaces for $m=1,2,\ldots$ (see Example \ref{examp:2}). 

In the following theorem we generalize the lower bound for the discrepancy of Levin's normal numbers from base $2$ to arbitrary prime base $p$.  

\begin{theorem}\label{thm:2}
Let $p\in\PP$. Let $\alpha\in[0,1)$ be such that its base $p$ representation is the concatenation of $p$-ary affine necklaces $\mathcal{A}_m$ for $m=1,2,3,\ldots$ with $\boldsymbol{z}=\boldsymbol{\eta}=\boldsymbol{0}$ and $\boldsymbol{u}=\boldsymbol{1}$. Then its star discrepancy $D_N^*$ satisfies 
$$ND^*_N(\{p^n\alpha\})\geq c_p(\log^2 N)$$
for infinitely many $N$, where $c_p>0$ is an absolute constant only depending on $p$.  
\end{theorem}

For the proof of Theorem \ref{thm:2}, that will be given in Section \ref{sec:5}, we collect important auxiliary results on the binomial coefficients and on the matrix $P^{(\boldsymbol{\eta},\boldsymbol{u})}$.

\begin{lemma}\label{lem:1a}
Let $p\in\PP$, $i$ and $u\in\mathbb{N}_0$. Then $\sum_{j=1}^u\binom{i+j}{j}\equiv\binom{i+1+u}{u}-1\pmod{p}$. 
\end{lemma}
\begin{proof}
This is an easy consequence of $\sum_{j=0}^u\binom{i+j}{j}=\binom{i+1+u}{u}$, which can be shown by induction on $u$. 
\end{proof}
We will heavily use the well-known Lucas Theorem, which we formulate in the following lemma.  
\begin{lemma}\label{lem:Lucas_Thm}
Let $p\in\PP$ and $m,n\in\NN_0$ with unique base $p$ representations $m=m_0+m_1p+m_2p^2+\cdots$ and $n=n_0+n_1p+n_2p^2+\cdots$. Then 
$$\binom{m}{n}\equiv \prod_{i=0}^\infty \binom{m_i}{n_i} \pmod{p}.$$
\end{lemma}
Lemma~\ref{lem:Lucas_Thm} explains a self-similar structure in $$P^{(\boldsymbol{0},\boldsymbol{1})}=\left(\binom{i+j}{j}\pmod{p}\right)_{i,j\geq 0}=:(p_{i,j})_{i,j\geq 0},$$ as $p_{pi,pj}=p_{i,j}$ or $p_{pi+i_0,pj+i_0}\equiv p_{i,j}p_{i_0,j_0}\pmod{p}$ where $i_0,j_0\in\{0,1,\ldots,p-1\}$. 

This self-similar structure guarantees a number of $0$ elements in specific square submatrices in $P^{(\boldsymbol{0},\boldsymbol{1})}$ that is quantified in Lemma \ref{lem:1b}, which will be fitting in the proof of Theorem \ref{thm:2} in the subsequent section. 

\begin{lemma}\label{lem:1b}
Let $p\in\PP$, $m>7$, and
$$D_m:=\left(\binom{i+1+j}{j}-1\pmod{p}\right)_{i,j}$$
with $p^{m-3}-p(p^3-1) p^{m-7}\leq i< p^{m-3}$ and $ 0\leq j< p^{m-7}$
The relative number of $(p-1)$s in $D_m$ is $\left(1-\left(\frac{p+1}{2p}\right)^{m-7}\right)$.
\end{lemma}
\begin{proof}
First we see as a consequence of Lucas' Theorem, that $\binom{p^r-1+1+j}{j}\equiv\binom{1}{0}\binom{j}{j}\pmod{p}=1=\binom{0+j}{j}$ for $0\leq j<p^r$. Hence the number of $0$s in 
\begin{equation}\label{eq:submat1}
\left(\binom{i+1+j}{j}\pmod{p}\right)_{0\leq i<p^r,0\leq j<p^r}
\end{equation}

equals the number of $0$s in 
\begin{equation}\label{eq:submat2}
\left(\binom{i+j}{j}\pmod{p}\right)_{0\leq i<p^r,0\leq j<p^r}=\left(\binom{i+j}{i}\pmod{p}\right)_{0\leq i<p^r,0\leq j<p^r},
\end{equation}
or equivalently the number of $0$s in

$$\left(\binom{j}{i}\pmod{p}\right)_{0\leq i<p^r,i\leq j<p^r+i}.$$

Since $\binom{p^r+l}{i}\equiv\binom{l}{i}\pmod{p}$ for all $0\leq l,i <p^r$ the latter equals the number of $0$s in
\begin{equation}\label{eq:submat3}
\left(\binom{j}{i}\pmod{p}\right)_{0\leq i<p^r,0\leq j<p^r}.
\end{equation}
See Figure \ref{fig:submatrices} for an illustration of the submatrices for $p=2$ and $r=3$ and for $p=3$ and $r=2$. 

\begin{figure}[h]
	\centering
			\includegraphics[width=0.12\textwidth]{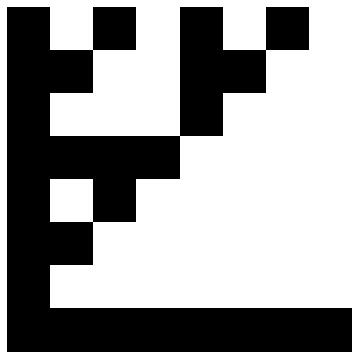}\hspace{0.9cm}	\includegraphics[width=0.12\textwidth]{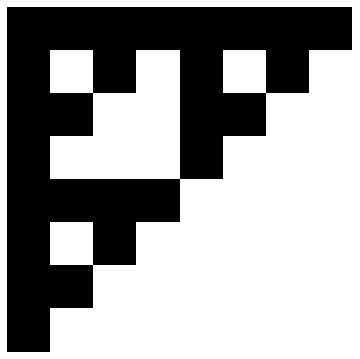}\hspace{0.9cm}	\includegraphics[width=0.12\textwidth]{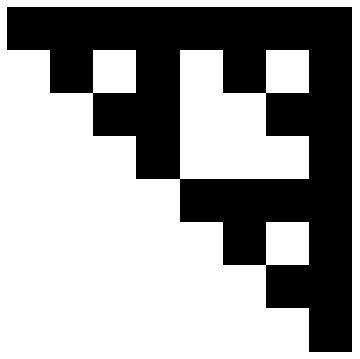}\\
		\includegraphics[width=0.12\textwidth]{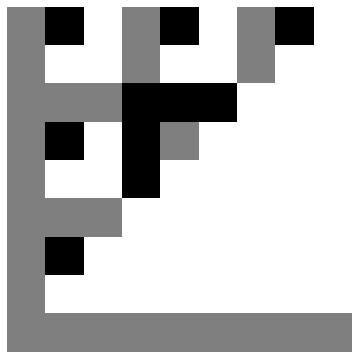}\hspace{0.9cm}	\includegraphics[width=0.12\textwidth]{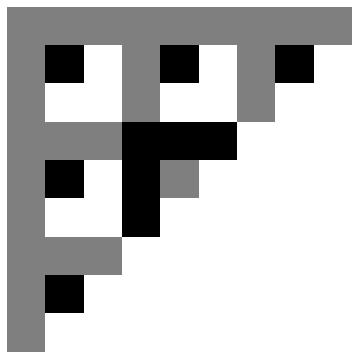}\hspace{0.9cm}	\includegraphics[width=0.12\textwidth]{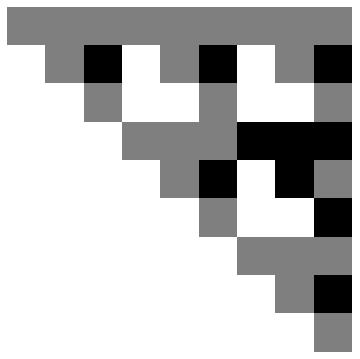}
	\caption{The matrices in \eqref{eq:submat1}, \eqref{eq:submat2}, and \eqref{eq:submat3} for $p=2$ and $r=3$ (top).
	The matrices in \eqref{eq:submat1}, \eqref{eq:submat2}, and \eqref{eq:submat3} for $p=3$ and $r=2$ (bottom). }
	\label{fig:submatrices}
\end{figure}

The number of nonzero elements in the latter matrix can be computed as $(\frac{p(p+1)}{2})^r$. This is an easy consequence of Lucas Theorem and the fact that the number of nonzero elements in the upper left $p\times p$ submatrix is $\sum_{i=1}^{p}i=\frac{p(p+1)}{2}$. Note, $\binom{j}{i}\not\equiv0\pmod{p}$ if $p>j\geq i\geq 0$ and $\binom{j}{i}\equiv0\pmod{p}$ else. 

The statement of the lemma follows then again by the self-similar structure of $D_m+1\pmod{p}$, that is built by a submatrix with $r=m-7$, as described above in \eqref{eq:submat1}, stacked $p(p^3-1)$ times. For the selfsimilar structure note $\binom{i+a p^r+j}{i+ap^r}\equiv\binom{i+j}{j}\pmod{p}$ for $0\leq i+j<p^r$ and $a\in\ZZ$ and $\binom{i+a p^r+j}{i+ap^r}\equiv\binom{i+j-p^r}{j}\binom{a+1}{a}\equiv0\pmod{p}$ for $p^r\leq i+j<2p^r$ as then obviously $i+j-p^r<i$. See Figure \ref{fig:D_m} for the transpose of $D_m+1\pmod{p}$ for $m=10$ and $p=2$ and for $m=9$ and $p=3$.
\begin{figure}[h]
	\centering
		\includegraphics[width=0.99\textwidth]{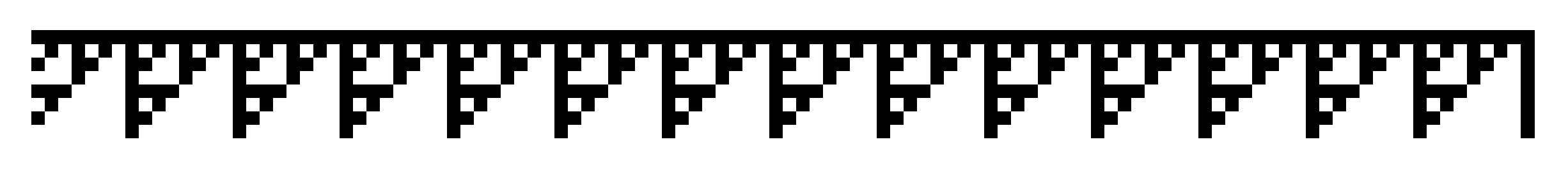}\\
		\includegraphics[width=0.99\textwidth]{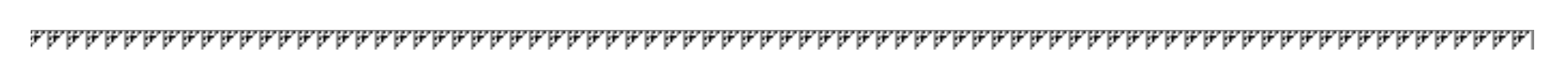}
	\caption{The transpose of $D_{10}+1\pmod{p}$ with $p=2$ (top) and $D_{9}+1\pmod{p}$ with $p=3$ (bottom). }
	\label{fig:D_m}
\end{figure}

\end{proof}

For stating further specific properties in $P^{(\boldsymbol{\eta},\boldsymbol{u})}$, which are relevant for the distribution of $(\{p^n\alpha\})_{n=0,1,\ldots}$, we introduce a notation for specific submatrices that are sketched in Figure \ref{fig:submatrices}.

Let $t\in\NN$ and $k\in\NN_0$ such that $t+k<p^m$. Set
$$A^{(\boldsymbol{\eta},\boldsymbol{u})}_{k,t}:=\left(\binom{i+j-\eta_j}{j}u_j\right)_{i=k,\ldots,k+t-1,\,j=0,\ldots, t-1},$$ 
i.e. the $t\times t$ submatrix of $P^{(\boldsymbol{\eta},\boldsymbol{u})}$ that selects the first $t$ columns of the $k$th up to the $k$th plus $(t-1)$th rows of $P^{(\boldsymbol{\eta},\boldsymbol{u})}$. 
Furthermore, set
$$B^{(\boldsymbol{\eta},\boldsymbol{u})}_{k,t}=\left(\binom{i+j-\eta_j}{j}u_j\right)_{i=k,\ldots,k+t-1,\,j=t,\ldots,p^m-1},$$ 
i.e. the $t\times (p^m-t)$ submatrix in $P^{(\boldsymbol{\eta},\boldsymbol{u})}$ that is located just to the right of $A^{(\boldsymbol{\eta},\boldsymbol{u})}_{k,t}$. \\
Let $t\in\NN$ and $k\in\NN_0$ such that $t+k<p^m$, let 
$$c^{(\boldsymbol{\eta},\boldsymbol{u})}_{k+t,t}:=\left(\binom{k+t+j-\eta_j}{j}u_j\right)_{j=0,\ldots, t-1}\quad\mbox{and}\quad d^{(\boldsymbol{\eta},\boldsymbol{u})}_{k+t,t}:=\left(\binom{k+t+j-\eta_j}{j}u_j\right)_{j=t,\ldots, p^m-1}.$$

\begin{figure}[h]
	\centering
		\includegraphics[width=0.90\textwidth]{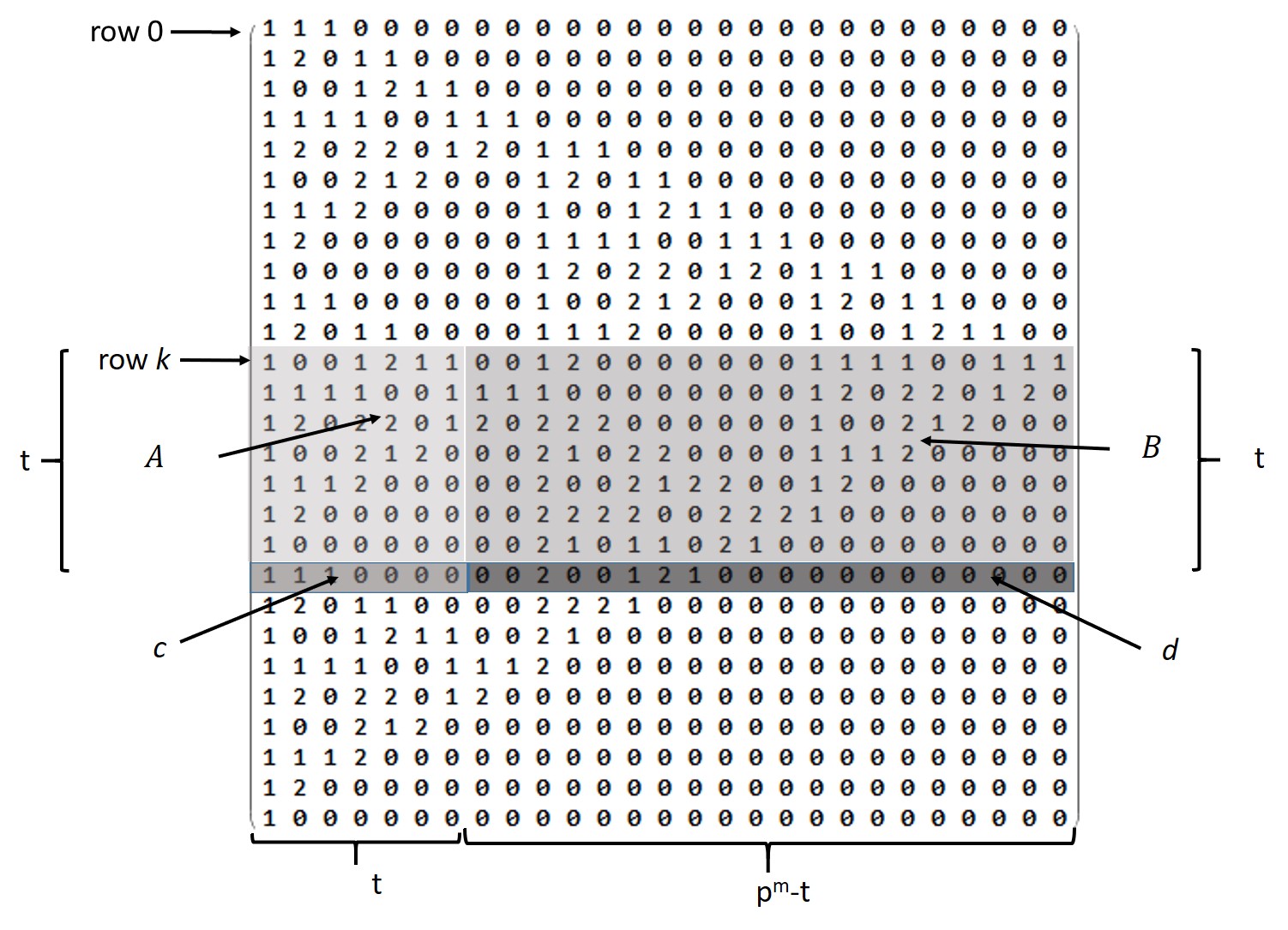}
	\caption{The submatrices $A^{(\boldsymbol{\eta},\boldsymbol{u})}_{k,t}=:A$, $B^{(\boldsymbol{\eta},\boldsymbol{u})}_{k,t}=:B$, $c^{(\boldsymbol{\eta},\boldsymbol{u})}_{k+t,t}=:c$, and $d^{(\boldsymbol{\eta},\boldsymbol{u})}_{k+t,t}=:d$, for $p=3$, $m=3$, $\eta_j=\lfloor \frac{j}{\sqrt{5}}\rfloor$, $u_j=1$, $t=7$, and $k=11$.}
	\label{fig:ABcd}
\end{figure}

For $t\in\NN$ we define $\xi_t := \left((-1)^{t+1}\binom{t}{0}, (-1)^{t+2}\binom{t}{1}, \ldots, (-1)^{2t}\binom{t}{t-1}\right)$. 
Then we immediately can extend the statements of \cite[Lemma~6, Lemma~7, and Corollary~1]{Lar}. Although, Theorem \ref{thm:2} only addresses the case $\boldsymbol{\eta}=\boldsymbol{0}$ and $\boldsymbol{u}=\boldsymbol{1}$ we put here these extended statements for general $\boldsymbol{\eta}$ and $\boldsymbol{u}$ to be able to discuss possible generalizations of Theorem \ref{thm:2}. See Remark \ref{rem:3}. 

\begin{lemma} \label{lem_c}Let $p\in\PP$. For $k,t$ such that $0\leq k<p^m$, $1\leq t\leq p^m$ and $k+t\leq p^m-1$ we have
\begin{enumerate}
	\item $
c^{(\boldsymbol{\eta},\boldsymbol{u})}_{k+t,t} \equiv \xi_t\cdot A^{(\boldsymbol{\eta},\boldsymbol{u})}_{k,t}\pmod{p},
$
\item $
d^{(\boldsymbol{\eta},\boldsymbol{u})}_{k+t,t} \equiv \xi_t\cdot B^{(\boldsymbol{\eta},\boldsymbol{u})}_{k,t}+c^{(\boldsymbol{\eta}^{(t)},\boldsymbol{u}^{(t)})}_{k+t,p^m-t}\pmod{p},
$ where $\boldsymbol{\eta}^{(t)}=(\eta_t,\eta_{t+1},\eta_{t+2}, \cdots)$ and $\boldsymbol{u}^{(t)}=(u_t,u_{t+1},u_{t+2},\ldots)$. 
\end{enumerate}

\end{lemma}

For the proof of Lemma~\ref{lem_c} we will need the following identity.
\begin{lemma} \label{lem_d}
For all non-negative integers $t, k$ we have
\begin{equation*}
\sum_{i=\max(\eta_j-k,0)}^t (-1)^{t+i+1}\binom{t}{i} \cdot \binom{k+j+i-\eta_j}{j} = -\binom{k+j-\eta_j}{j-t}, 
\end{equation*}
\end{lemma}
\begin{proof}
This is simple induction on $t$ and on $r=\max(\eta_j-k,0)$ together with the basic facts that $\binom{t}{i}=\binom{t-1}{i}+\binom{t-1}{i-1}$, therefore $\binom{n}{k}-\binom{n-1}{k}=\binom{n-1}{k-1}$, $\binom{k}{j}=0$ whenever $k<j$ or $j\leq 0$. 
\end{proof}
From Lemma~\ref{lem_d}, we immediately conclude: 
\begin{corollary} \label{cor_b}
~\
\begin{enumerate}
\item [(a)] $\sum_{i=\max(\eta_j-k,0)}^t (-1)^{t+i+1}\binom{t}{i} \cdot \binom{k+j+i-\eta_j}{j} = 0$ for $j = 0,1,\ldots, t-1$,
\item [(b)] $\sum_{i=\max(\eta_j-k,0)}^{t-1} (-1)^{t+i+1}\binom{t}{i} \cdot \binom{k+j+i-\eta_j}{j} =  \binom{k+j+t-\eta_j}{j}$ for $j=0,1,\ldots, t-1$.
\item [(c)] $\sum_{i=\max(\eta_j-k,0)}^{t-1} (-1)^{t+i+1}\binom{t}{i} \cdot \binom{k+j+i-\eta_j}{j} =  \binom{k+j+t-\eta_j}{j}-\binom{k+j-\eta_j}{j-t}$ \newline 
for $j=t,t+1,\ldots, p^m-1$.
\end{enumerate}

\end{corollary}

\emph{Proof of Lemma \ref{lem_c}}. We start with Item (1): $
c^{(\boldsymbol{\eta},\boldsymbol{u})}_{k+t,t} \equiv \xi_t\cdot A^{(\boldsymbol{\eta},\boldsymbol{u})}_{k,t}\pmod{p}$. This is equivalent to $\sum_{i=\max(\eta_j-k,0)}^{t-1} (-1)^{t+i+1}\binom{t}{i} \cdot \binom{k+j+i-\eta_j}{j} =  \binom{k+j+t-\eta_j}{j}$ for $j=0,1,\ldots, t-1$. This is exactly Corollary \ref{cor_b}, Item (b). 

For the proof of Item (2), i.e. $
d^{(\boldsymbol{\eta},\boldsymbol{u})}_{k+t,t} \equiv \xi_t\cdot B^{(\boldsymbol{\eta},\boldsymbol{u})}_{k,t}+c^{(\boldsymbol{\eta}^{(t)},\boldsymbol{u}^{(t)})}_{k+t,2^m-t}\pmod{p},
$ we see the equivalence to Item (c) in Corollary \ref{cor_b}: $\sum_{i=\max(\eta_j-k,0)}^{t-1} (-1)^{t+i+1}\binom{t}{i} \cdot \binom{k+j+i-\eta_j}{j} =  \binom{k+j+t-\eta_j}{j}-\binom{k+j-\eta_j}{j-t}$ for $j=t,t+1,\ldots, p^m-1$.
\hfill$\qed$


\section{Proof of Theorem \ref{thm:2} and Remark \ref{rem:3}}\label{sec:5}

Now we are ready to prove Theorem \ref{thm:2} and  to discuss its generalisability in Remark \ref{rem:3}. \\

\emph{Proof of Theorem \ref{thm:2}}. 

We make use of the notation $n_m = \sum_{j=1}^{m-1} p^j p^{p^j}$ for $m\in\NN$. For every $m$ large enough, we will show that there exists an $N$ with  $n_m < N < n_{m-1}$ and an interval $J:=J_m \subseteq [0,1)$ such that 
\begin{equation}\label{eq:tomany}
\# \left\{0 \leq n < N \left| ~\left\{p^n \alpha\right\} \in J\right\} \right. \geq N \cdot \lambda (J) + c_p \cdot \left(\log N\right)^2
\end{equation}
with a fixed (i.e. independent of $m$) absolute positive constant $c_p$ (depending only on $p$), and where $\lambda (J)$ denotes the length of the interval $J$. This proves Theorem \ref{thm:2}. \\
We can proceed in principle quite analogous as in the proof \cite[Theorem 1]{Lar}, by applying the generalizations of the needed auxiliary results, as e.g. Lemma \ref{lem:1b}, \ref{lem_c} or \ref{lem:2gen} etc., and elaborating some necessary technical adaptions.\\

For given $m$ (large enough) let $w_l := p^{m-3}-1-p^3l$ for $l =0,1,\ldots, p^{m-7} - 1 =: M.$\\
For $m \geq 7$ all $w_l$ and $M$ are integers. The $w_l$ are congruent $(p-1)$ modulo $p$, and we have \\

$$ 
(p-1)p^{m-4}<w_M<w_{M-1}<\ldots<w_0<p^{m-3}. 
$$\\

Let $N := n_m + p^m \cdot \left( p^{w_0}+ \ldots +p^{w_{M-1}} + p^{w_{M}}\right)$.\\

Note also that  $N < n_{m+1} < p^{m+1} \cdot p^{p^m}$. Hence  $p^m \ge  \frac{\log N}{p \cdot \log p}$ for $m$ large enough. \\

We will consider the sequence of elements $x_n := \left\{p^n \alpha\right\}$ for $n=0,1,\ldots, N-1$, i.e., the points $x_0, \ldots, x_{n_{m}-1}$ and the points
$x_{n,k} := \left\{p^{(n_m + p^m \cdot n + k)} \cdot \alpha \right\}$
for $n=0,1,\ldots, p^{w_M} + p^{w_{M-1}} + \ldots + p^{w_0}-1$ and $k=0,1,\ldots, p^m-1$. We divide the latter set of $n$'s in blocks $\mathcal{B}_l$ of the form $n = B_l, B_l +1, \ldots, B_l +p^{w_l} -1$ for $l = 0, \ldots, M$, where $B_0 := 0$ and $B_l := p^{w_0} + \ldots + p^{w_{l-1}}$.\\

We construct in the following an interval $J \subseteq [0,1)$ that contains ``too many'' of the points $x_0, \ldots, x_{N-1}$, as quanitified in \eqref{eq:tomany}. $J$ will be of the form $J_M \cup J_{M-1} \cup \ldots \cup J_0$ with $J_l := \left[\frac{U(l)}{p^{w_l}}, \frac{V(l)}{p^{w_l}}\right)$, where $0 \leq U(l) < V(l) < p^{w_l}, \quad U(l)\in\NN_0, V(l) \in \frac{1}{p} \mathbb{N}_0$ and with $0< V(l) - U(l) < 2$, and $\frac{V(l)}{p^{w_l}} = \frac{U(l-1)}{p^{w_{l-1}}}$ for $l = 1,\ldots, M$. That means, $J$ is of the form as sketched in Figure \ref{fig:3}.\\

\begin{figure}
\includegraphics[angle=0,width=120mm]{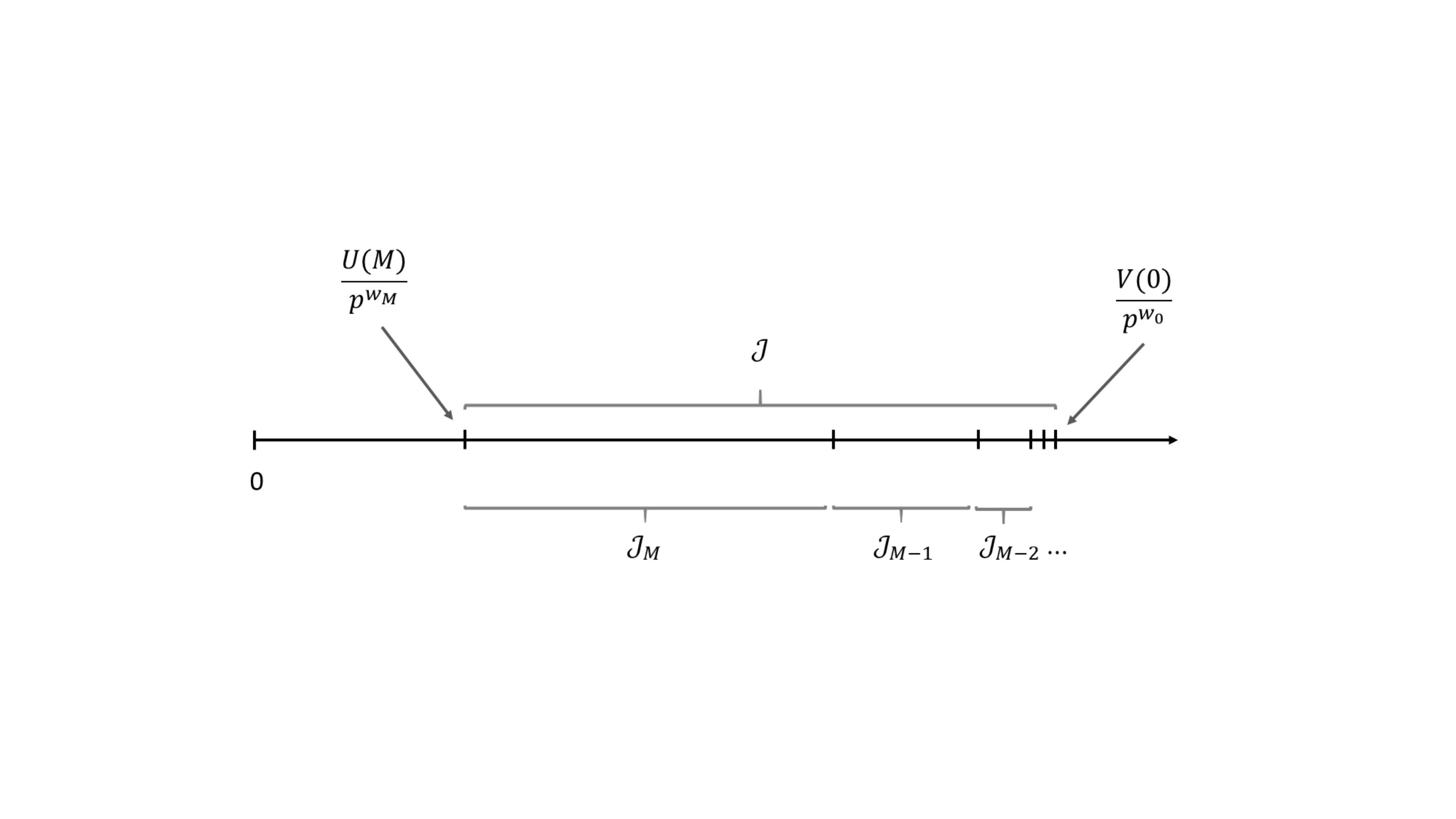} 
~\\ 
\caption{The interval $J$.}\label{fig:3}
\end{figure}
For the length $\lambda(J)$ of the interval $J$ we have
$$
\lambda(J) \leq 2 \cdot \sum^M_{l=0} \frac{1}{p^{w_l}} \leq 4 \cdot \frac{1}{p^{w_M}} \leq \frac{4}{p^{p^{m-4}}},
$$
where we used $w_M>p^{m-4}$. \\

Now let us utilize Lemma \ref{lem:2gen} with $b := p$: 

For every $l = 0,\ldots, M$ we consider the points $x_{n,k}$ with $0 \leq k < p^m$ and $n \in \mathcal{B}_l$. By Lemma \ref{lem:2gen} there are at most $2 p^m$ integers $c_l^{(i)}$ such that the interval $\left[\frac{c_l^{(i)}}{p^{w_l}}, \frac{c^{(i)}_{l} +1}{p^{w_l}}\right)$ does not contain exactly $p^m$ points of these $x_{n,k}$.\\
Altogether there are at most $2 p^m \cdot (M+1) \leq  p^{2m-6}$ intervals of the form $\left[\frac{c_l^{(i)}}{p^{w_l}}, \frac{c^{(i)}_{l} +1}{p^{w_l}}\right)$ which do not contain exactly $p^m$ of the $x_{n,k}$ with $0 \leq k < p^m$ and $n \in \mathcal{B}_l$ for some $l = 0,\ldots, M$. So there are at most $p^{2m-6}$ such ``exceptional intervals'', and the length of the union of these intervals is at most $p^{2m-6} \cdot \sum^M_{l=0} \frac{1}{p^{w_l}} < \frac{p^{2m-5}}{p^{w_M}}< \frac{p^{2m-5}}{p^{p^{m-4}}}< \frac{1}{4}$ for $m$ large enough.\\
Hence there exists a sub-interval $Z$ of $[0,1)$ with length at least $\frac{3}{4}\frac{1}{p^{2 m-6}}$ which has empty intersection with every of the exceptional intervals. 

Let $U(M)$ be the least integer such that $U(M)$ is not congruent to $p-1$ modulo $p$ and such that $\frac{U(M)}{p^{w_M}} \in Z$. For $m$ large enough such $U(M)$ certainly exist. This value $\frac{U(M)}{p^{w_M}}$ is the left border of $J_M$ and hence of $J$. Since $\lambda(J) + \frac{p}{p^{w_M}} \leq \frac{4+p}{p^{w_M}} \leq \frac{4+p}{p^{p^{m-4}}}$ and $\lambda (Z) \geq \frac{3}{4} \cdot \frac{1}{p^{2 m-6}}$, for $m$ large enough we have $J \subseteq Z$ and hence $J$ has empty intersection with every of the exceptional intervals.\\

We start by constructing $J_M$: For this reason we consider the points $x_{n,k}$ for $n \in \mathcal{B}_M$ and $0 \leq k < p^m$. Let $\tilde{J}_M := \left[\frac{U(M)}{p^{w_M}}, \frac{U(M) +1}{p^{w_M}}\right)$. \\

We will show now that for each $k\geq 0$ with $0 \leq k+w_M <  p^m$ there is exactly one $n \in \mathcal{B}_M$ such that $x_{n,k} \in \tilde{J}_M$, and in a second step we will analyze in which sub-interval
$$
J_{M, \gamma} := \left[\frac{U(M)}{p^{w_M}} + \frac{\gamma}{p^{w_M +1}}, \frac{U(M)}{p^{w_M}} + \frac{\gamma+1}{p^{w_M +1}}\right)
$$
for $\gamma = 0,1, \ldots, p-1$ this $x_{n,k}$ is located. \\

Now let $k\geq 0$ with $ k+w_M <  p^m$. We write $U(M) := u_{w_M-1} + u_{w_M-2} \cdot p + \cdots + u_0 \cdot p^{w_M-1}$, 
\begin{eqnarray*}
n&=&e_0(n) + p \cdot e_1(n) + \cdots + p^{w_M-1} \cdot e_{w_M-1}(n)+B_{M}\\
&=& e_0(n) + p \cdot e_1(n) + \cdots + p^{w_M-1} \cdot e_{w_M-1}(n)+p^{w_{M-1}}+p^{w_{M-2}}+\cdots+p^{w_0}\\
&=&e_0(n) + p \cdot e_1(n) + \cdots + p^{w_M-1} \cdot e_{w_M-1}(n)+p^{w_M+p^3}+p^{w_M+2\cdot p^3}+\cdots+2^{w_M+M\cdot p^3}.
\end{eqnarray*}
Note that by the choice of $U(M)$, we know $u_{w_M-1}\neq p-1$. 

Then this leads to the following two systems, where $\overline{0}:=(0,0,0,\ldots,0,0)^T\in\FF_p^{p^3}$, $\overline{1}:=(1,0,0,\ldots,0,0)^T\in\FF_p^{p^3}$, and the vector consisting of $\overline{0}$s, $\overline{1}$s, and $0$s, contains $M$ consecutive $\overline{1}$s:
$$
A^{(\boldsymbol{0},\boldsymbol{1})}_{k,w_M} \cdot \left(\begin{tabular}{c} $e_0(n)$ \\ $\vdots$ \\ $e_{w_M-1} (n)$ \end{tabular}\right) + B^{(\boldsymbol{0},\boldsymbol{1})}_{k,w_M}\left(\begin{array}{c}\overline{0}\\\overline{1}\\\overline{1}\\\vdots\\\overline{1}\\\overline{1}\\0\\\vdots\\0\end{array}\right) \equiv \left(\begin{tabular}{c} $u_0$ \\ $\vdots$ \\ $u_{w_M -1}$ \end{tabular}\right)\pmod{p}
$$
and
$$
c^{(\boldsymbol{0},\boldsymbol{1})}_{k+w_M,w_M} \cdot \left(\begin{tabular}{c} $e_0 (n)$ \\ $\vdots$ \\ $e_{w_M-1}(n)$ \end{tabular}\right) + d^{(\boldsymbol{0},\boldsymbol{1})}_{k+w_M,w_M}\left(\begin{array}{c}\overline{0}\\\overline{1}\\\overline{1}\\\vdots\\\overline{1}\\\overline{1}\\0\\\vdots\\0\end{array}\right) \equiv \gamma \pmod{p}.
$$
Here $A^{(\boldsymbol{0},\boldsymbol{1})}_{k,w_M}$, $C^{(\boldsymbol{0},\boldsymbol{1})}_{k,w_M}$, $c^{(\boldsymbol{0},\boldsymbol{1})}_{k+w_M,w_M}$, and $d^{(\boldsymbol{0},\boldsymbol{1})}_{k+w_M,w_M}$ are the magnitudes defined in Section ~\ref{sec:4}. Note that here we used the fact that 
$k + w_M<p^m$. Otherwise the system would contain conditions described by using $e_j (n+1)$. 

Since by Item (2) of Lemma \ref{lem:rank} the matrix $A^{(\boldsymbol{0},\boldsymbol{1})}_{k,w_M}$ is regular and we know that the first system for every $k$ has a unique solution, i.e.,
$$
\left(\begin{tabular}{c} $e_0 (n)$ \\ $\vdots$ \\ $e_{w_M-1}$ \end{tabular}\right) \equiv (A^{(\boldsymbol{0},\boldsymbol{1})}_{k,w_M})^{-1} \cdot \left(\begin{tabular}{c} $u_0$ \\ $\vdots$ \\  $u_{w_M-1}$ \end{tabular}\right)-(A^{(\boldsymbol{0},\boldsymbol{1})}_{k,w_M})^{-1}B^{(\boldsymbol{0},\boldsymbol{1})}_{k,w_M}\left(\begin{array}{c}\overline{0}\\\overline{1}\\\overline{1}\\\vdots\\\overline{1}\\\overline{1}\\0\\\vdots\\0\end{array}\right)\pmod{p},
$$
 as announced already in the beginning of the construction of $J_M$. Inserting this unique solution into the second system yields
\begin{align*}
\gamma \equiv &\,\,\, c^{(\boldsymbol{0},\boldsymbol{1})}_{k+w_M,w_M} \cdot (A^{(\boldsymbol{0},\boldsymbol{1})}_{k,w_M})^{-1} \cdot \left(\begin{tabular}{c} $u_0$ \\ $\vdots$ \\ $u_{w_M-1}$ \end{tabular}\right) -c^{(\boldsymbol{0},\boldsymbol{1})}_{k+w_M,w_M} \cdot (A^{(\boldsymbol{0},\boldsymbol{1})}_{k,w_M})^{-1}B^{(\boldsymbol{0},\boldsymbol{1})}_{k,w_M}\left(\begin{array}{c}\overline{0}\\\overline{1}\\\overline{1}\\\vdots\\\overline{1}\\\overline{1}\\0\\\vdots\\0\end{array}\right)\\
& +  d^{(\boldsymbol{0},\boldsymbol{1})}_{k+w_M,w_M}\left(\begin{array}{c}\overline{0}\\\overline{1}\\\overline{1}\\\vdots\\\overline{1}\\\overline{1}\\0\\\vdots\\0\end{array}\right)\pmod{p}. 
\end{align*}
Lemma~\ref{lem_c}, Item (1) and Item (2), allow us to rewrite the above equation as 
\begin{equation}\label{equ:gamma_M}
\gamma \equiv \xi_{w_M}\cdot \left(\begin{tabular}{c} $u_0$ \\ $\vdots$ \\ $u_{w_M-1}$ \end{tabular}\right)+ c^{(\boldsymbol{0}^{(w_M)},\boldsymbol{1}^{(w_M)})}_{k+w_M,p^m-w_M} \left(\begin{array}{c}\overline{0}\\\overline{1}\\\overline{1}\\\vdots\\\overline{1}\\\overline{1}\\0\\\vdots\\0\end{array}\right)\pmod{p}. 
\end{equation}

Let $\mathcal{A}_{\delta}(M)$ be the number of $k$s in $\{0,\ldots, p^m-1-w_M\}$ such that 
$\gamma=\delta$ with fixed $\delta\in\{0,1,\ldots,p-1\}$. 
We define 
$$\mathcal{A}(M):=\max(\mathcal{A}_0(M),\mathcal{A}_1(M),\ldots,\mathcal{A}_{p-1}(M)),$$
which we will estimate later.

Altogether we know, that there is a $\gamma\in\{0,1,\ldots,p-1\}$ such that $x_{n,k}\in J_{M,\gamma}$ for at least $\mathcal{A}(M)=:q(M)p^m$ values of $k$s.\\

We distinguish between the cases $\gamma\neq p-1$ and $\gamma=p-1$: 

\textbf{If} $\boldsymbol{\gamma \neq p-1}$, then we choose $J_M := \left[\frac{U(M)}{p^{w_M}} , \frac{U(M) + \frac{p-1}{p}}{p^{w_M}}\right)$. $J_M$ then contains at least $\mathcal{A}(M)=q(M)p^m$ of the points $x_{n,k}$ with $n \in \mathcal{B}_M$. \\

\textbf{If} $\boldsymbol{\gamma = p-1}$, then we choose $J_M := \left[\frac{U(M)}{p^{w_M}} , \frac{U(M) + \frac{2p-1}{p}}{p^{w_M}}\right)$. The reason for this choice is the following: Since $\tilde{J}_M \subseteq Z$ the interval $\tilde{J}_M=\left[\frac{U(M)}{p^{w_M}} , \frac{U(M) + 1}{p^{w_M}}\right)$ is not an ``exceptional interval'' and contains exactly $p^m$ points $x_{n,k}$ with $0 \leq k < p^m$ and $n \in \mathcal{B}_M$.\\

The question, how many points of $x_{n,k}$ with $ k+w_M <   p^m$  lie in 
$\left[\frac{U(M) + 1}{p^{w_M}}, \frac{U(M) + 2}{p^{w_M}}\right)$, which again is not an ``exceptional interval'' and, more detailed, in which of the sub-intervals $\left[\frac{U(M) +1}{p^{w_M}} + \frac{\tilde{\gamma}}{p^{w_M+1}}, \frac{U(M) +1}{p^{w_M}} + \frac{\tilde{\gamma}+1}{p^{w_M+1}}\right)$ with $\tilde{\gamma}\in\{0,1,\ldots,p-1\}$ these points are located, now leads to the systems
$$
A^{(\boldsymbol{0},\boldsymbol{1})}_{k,w_M} \cdot \left(\begin{tabular}{c} $e_0(n)$ \\ $\vdots$ \\ $e_{w_M-1} (n)$ \end{tabular}\right) + B^{(\boldsymbol{0},\boldsymbol{1})}_{k,w_M}\left(\begin{array}{c}\overline{0}\\\overline{1}\\\vdots\\\overline{1}\\0\\\vdots\\0\end{array}\right) \equiv \left(\begin{tabular}{c} $u_0$ \\ $\vdots$ \\ $u_{w_M -1}+1$ \end{tabular}\right)\pmod{p}
$$
and
$$
c^{(\boldsymbol{0},\boldsymbol{1})}_{k+w_M,w_M} \cdot \left(\begin{tabular}{c} $e_0 (n)$ \\ $\vdots$ \\ $e_{w_M-1}(n)$ \end{tabular}\right) + d^{(\boldsymbol{0},\boldsymbol{1})}_{k+w_M,w_M}\left(\begin{array}{c}\overline{0}\\\overline{1}\\\vdots\\\overline{1}\\0\\\vdots\\0\end{array}\right)  \equiv \tilde{\gamma}\pmod{p}.
$$
Hence, as before
\begin{align*}
\tilde{\gamma} \equiv & {c^{(\boldsymbol{0},\boldsymbol{1})}_{k+w_M,w_M} \cdot (A^{(\boldsymbol{0},\boldsymbol{1})}_{k,w_M})^{-1}} \cdot \left(\begin{tabular}{c} $u_0$ \\ $\vdots$ \\ $u_{w_M-1}+1$ \end{tabular}\right) -c^{(\boldsymbol{0},\boldsymbol{1})}_{k+w_M,w_M} \cdot (A^{(\boldsymbol{0},\boldsymbol{1})}_{k,w_M})^{-1}B^{(\boldsymbol{0},\boldsymbol{1})}_{k,w_M}\left(\begin{array}{c}\overline{0}\\\overline{1}\\\vdots\\\overline{1}\\0\\\vdots\\0\end{array}\right)\\
& +  d_{k+w_M,w_M}\left(\begin{array}{c}\overline{0}\\\overline{1}\\\vdots\\\overline{1}\\0\\\vdots\\0\end{array}\right)\pmod{p}.
\end{align*}
or equivalently using again Lemma \ref{lem_c}

\begin{equation}\label{equ:tildegamma_M}
\tilde{\gamma} \equiv \xi_{w_M}\cdot \left(\begin{tabular}{c} $u_0$ \\ $\vdots$ \\ $u_{w_M-1}+1$ \end{tabular}\right)+ c^{(\boldsymbol{0}^{(w_M)},\boldsymbol{1}^{(w_M)})}_{k+w_M,p^m-w_M} \left(\begin{array}{c}\overline{0}\\\overline{1}\\\vdots\\\overline{1}\\0\\\vdots\\0\end{array}\right)\pmod{p}. 
\end{equation}

Note that we have chosen $U(M)$ to be an integer $\not\equiv p-1 \pmod{p}$, and so $u_{w_M-1}\neq p-1$. Moreover, by the definition of $\xi_t$ the last entry in $\xi_{w_M}$ equals $(-1)^{2w_M}\binom{w_M}{w_M-1}=w_M\equiv p-1\pmod{p}$ as we have chosen $w_l\equiv p-1\pmod{p}$ for all $l\in\{0,\ldots,M\}$. Therefore $\tilde{\gamma} \equiv \gamma+\xi_{w_M}\equiv\gamma +p- 1 \equiv p-1-1\equiv p-2 \pmod{p}$. Hence $\left[\frac{U(M) + 1}{p^{w_M}}, \frac{U(M) + 1}{p^{w_M}} + \frac{p-1}{p} \cdot \frac{1}{p^{w_M}}\right)$ contains at least $\mathcal{A}(M)=q(M) \cdot p^m$ points of the $x_{n,k}$ with $k+w_M < p^m$ and $n \in \mathcal{B}_M$.\\

We summarize: For both choices of $J_M$ we have
$$
\# \left\{(n,k)\left|n \in \mathcal{B}_M, 0 \leq k < p^m, x_{n,k} \in J_M\right\}\right. \geq p^m \cdot p^{w_M} \cdot \lambda \left(J_M\right) + \left(q(M)-\frac{p-1}{p}\right) \cdot p^m.
$$
Note that $V(M)<p^{w_M}$ is also guaranteed in both cases. \\

In the next step we will show how to choose the interval $J_{M-1}$. Then it will be clear how we will, quite analogously, choose the intervals $J_{M-2}, \ldots, J_0$.\\

We recall that the interval $J_{M-1}$ is denoted as $J_{M-1} = \left[\frac{U(M-1)}{p^{w_{M-1}}}, \frac{V(M-1)}{p^{w_{M-1}}}\right)$, where $\frac{U(M-1)}{p^{w_{M-1}}} = \frac{V(M)}{p^{w_M}}$. Note that ${w_{M-1}}={w_M}+p^3$. Thus $U(M-1)$ is a multiple of $p$ and therefore $\not\equiv p-1\pmod{p}$. We write again $U(M-1)$ in its unique base $p$ representation, $U(M-1) = u_{w_{M-1}-1} + u_{w_{M-1}-2} \cdot p + \ldots + u_0 \cdot p^{w_{M-1}-1}$. Similarly to the choice of $J_M$ we will choose $J_{M-1}$ either as
$$
J_{M-1} := J_{M-1}^{(1)} := \left[\frac{U(M-1)}{p^{w_{M-1}}}, \frac{U(M-1)}{p^{w_{M-1}}} + \frac{p-1}{p} \cdot \frac{1}{p^{w_{M-1}}}\right)
$$
or
$$
J_{M-1} := J_{M-1}^{(2)} := \left[\frac{U(M-1)}{p^{w_{M-1}}}, \frac{U(M-1)}{p^{w_{M-1}}} + \frac{2p-1}{p} \cdot \frac{1}{p^{w_{M-1}}}\right).
$$
To decide, which of the two choices we prefer, we first consider the interval 
$$
\tilde{J}_{M-1} := \left[\frac{U(M-1)}{p^{w_{M-1}}}, \frac{U(M-1)}{p^{w_{M-1}}} + \frac{1}{p^{w_{M-1}}}\right)
$$
and the points $x_{n,k}$ with $n \in \mathcal{B}_{M-1}$ and $k = 0,1,\ldots, p^m -1$. \\
$\tilde{J}_{M-1}$ by Lemma~\ref{lem:2gen} and Proposition \ref{prop:1} contains exactly $p^m$ of these points. Again for $k\geq 0$ such that $ k +w_{M-1}<  p^m$ we ask where exactly these points are located in $\tilde{J}_{M-1}$. Especially, again we ask how these points are distributed to the sub-intervals
$$
\tilde{J}_{M-1, \gamma} := \left[\frac{U(M-1)}{p^{w_{M-1}}} + \frac{\gamma}{p^{w_{M-1}+1}}, \frac{U(M-1)}{p^{w_{M-1}}} + \frac{\gamma+1}{p^{w_{M-1}+1}}\right)
$$
for $\gamma = 0, 1,\ldots,p-1$. Then for $n = e_0(n) + p \cdot e_1(n) + \ldots + p^{w_{M-1}-1} \cdot e_{w_{M-1}-1}(n)+B_{M-1}$ we arrive at the systems, where the vector consisting of $\overline{0}$,s, $\overline{1}$s, and $0$'s, now contains $(M-1)$ consecutive $\overline{1}$s:

$$
A^{(\boldsymbol{0},\boldsymbol{1})}_{k,w_{M-1}} \cdot \left(\begin{tabular}{c} $e_0(n)$ \\ $\vdots$ \\ $e_{w_{M-1}-1} (n)$ \end{tabular}\right) + B^{(\boldsymbol{0},\boldsymbol{1})}_{k,w_{M-1}}\left(\begin{array}{c}\overline{0}\\\overline{1}\\\vdots\\\overline{1}\\0\\\vdots\\0\end{array}\right) \equiv \left(\begin{tabular}{c} $u_0$ \\ $\vdots$ \\ $u_{w_{M-1} -1}$ \end{tabular}\right)\pmod{p}
$$
and
$$
c^{(\boldsymbol{0},\boldsymbol{1})}_{k+w_{M-1},w_{M-1}} \cdot \left(\begin{tabular}{c} $e_0 (n)$ \\ $\vdots$ \\ $e_{w_{M-1}-1}(n)$ \end{tabular}\right) + d^{(\boldsymbol{0},\boldsymbol{1})}_{k+w_{M-1},w_{M-1}}\left(\begin{array}{c}\overline{0}\\\overline{1}\\\vdots\\\overline{1}\\0\\\vdots\\0\end{array}\right) \equiv \gamma\pmod{p}.
$$
As in the construction of $J_M$ we solve the first system and insert to the second, to obtain after using Lemma \ref{lem_c}: 

\begin{equation*}
\gamma \equiv \xi_{w_{M-1}}\cdot \left(\begin{tabular}{c} $u_0$ \\ $\vdots$ \\ $u_{w_{M-1}-1}$ \end{tabular}\right)+ c^{(\boldsymbol{0}^{(w_{M-1})},\boldsymbol{1}^{(w_{M-1})})}_{k+w_{M-1},p^m-w_{M-1}} \left(\begin{array}{c}\overline{0}\\\overline{1}\\\vdots\\\overline{1}\\0\\\vdots\\0\end{array}\right)\pmod{p}. 
\end{equation*}

Then again distinguishing between $\gamma\neq p-1$ and $\gamma=p-1$ we obtain in both cases $J_{M-1}$, such that 
\begin{align*}
\# \left\{(n,k)\left|n \in \mathcal{B}_{M-1}, 0 \leq k < p^m ,x_{n,k} \in J_{M-1}\right\} \right. \geq& \,\,\, p^m \cdot p^{w_{M-1}} \cdot \lambda \left(J_{M-1}\right) +\\
&+ \left(q(M-1)-\frac{p-1}{p}\right) \cdot p^m.
\end{align*}
Now important is the following: 
Further, since $J_M \cup J_{M-1} \subseteq Z$, we know that the interval 
$$
J_M = \left[\frac{U(M)}{p^{w_M}}, \frac{V(M)}{p^{w_M}}\right) = \left[\frac{U(M) \cdot p^{p^3}}{p^{w_{M-1}}}, \frac{V(M) \cdot p^{p^3}}{p^{w_{M-1}}}\right)
$$
satisfies that $U(M) \cdot p^{p^3}$ and $V(M) \cdot p^{p^3}$ are integers and therefore $J_M$ contains at least $p^m \cdot p^{w_{M-1}} \cdot \lambda \left(J_M\right)$ of the points $x_{n,k}$ with $k=0, \ldots, p^m-1$ and $n \in \mathcal{B}_{M-1}$. Together
\begin{align*}
\# \left\{(n,k) \left|n \in \mathcal{B}_{M-1}, 0 \leq k < p^m,x_{n,k} \in J_M \cup J_{M-1}\right\}\right. \geq&\,\,\, p^m  p^{w_{M-1}}  \lambda \left(J_M \cup J_{M-1}\right)+ \\
&+ \left(q(M-1)-\frac{p-1}{p}\right) \cdot p^m.
\end{align*}
In exactly this way we proceed to construct $J_{M-1}, \ldots, J_0$ such that finally for every $l=0, \ldots, M$ we have: 
\begin{align}\label{equ:number:B_l}
\# \left\{(n,k)\left| n \in \mathcal{B}_l, 0 \leq k < p^m, x_{n,k} \in J_M \cup \ldots \cup J_l\right.\right\} \geq& \quad \quad\end{align}
\begin{align*}
\quad \quad&p^m \cdot p^{w_l} \cdot \lambda \left(J_M \cup \ldots \cup J_l\right)  + \left(q(l)-\frac{p-1}{p}\right) \cdot p^m.
\end{align*}
We set $J := J_M \cup \ldots \cup J_0$. We estimate $\mathcal{N}:=\# \left\{0 \leq n < N \left|x_n \in J\right.\right\}$ from below: 
\begin{eqnarray*}
 \mathcal{N} &=& \# \left\{0 \leq n < n_m\left|x_n \in J\right.\right\} \\
&&+ \sum_{l=0}^{M}\# \left\{(n,k)\left|n \in \mathcal{B}_l, 0 \leq k < p^m ,x_{n,k} \in J_M \cup \ldots \cup J_l\right.\right\}\\
&&+ \sum_{l=0}^{M}\# \left\{(n,k)\left|n \in \mathcal{B}_l, 0 \leq k < p^m, x_{n,k} \in J_{l-1} \cup \ldots \cup J_0\right.\right\}\\
&\geq &\# \left\{0 \leq n < n_m\left|x_n \in J\right.\right\} \\
&&+ \sum_{l=0}^{M}\# \left\{(n,k) \left|n \in \mathcal{B}_l, 0 \leq k < p^m, x_{n,k} \in J_M \cup \ldots \cup J_l\right.\right\}\\
&\geq &\# \left\{0 \leq n < n_m\left|x_n \in J\right.\right\}\\
&&+ \sum^M_{l=0} p^m \cdot p^{w_l} \cdot \lambda \left(J_M \cup \ldots \cup J_l\right) + \sum^M_{l=0}  \left(q(l)-\frac{p-1}{p}\right) \cdot p^m\\
&=& \# \left\{0 \leq n < n_m\left|x_n \in J\right.\right\}+(N-n_m)\lambda(J)\\
&&-\sum^M_{l=0} p^m \cdot p^{w_l} \cdot \lambda \left(J_{l-1} \cup \ldots \cup J_0\right) + \sum^M_{l=0}  \left(q(l)-\frac{p-1}{p}\right) \cdot p^m,
\end{eqnarray*}
where in the penultimate step we used \eqref{equ:number:B_l}. 

We derive an upper bound for $\lambda \left(J_{l-1} \cup \ldots \cup J_0\right)$:
\begin{equation}\label{eq:J_letc}
\lambda \left(J_{l-1} \cup \ldots \cup J_0\right)\leq \frac{2p-1}{p}\left(\frac{1}{p^{w_{l-1}}}+\cdots+\frac{1}{p^{w_0}}\right)\leq \frac{2p-1}{p}\frac{1}{p^{w_{l-1}}}\frac{1}{1-\frac1p}=\frac{2p-1}{(p-1)p^{p^3}}\frac{1}{p^{w_{l}}}.
\end{equation}
Furthermore, we know from Example \ref{examp:1}, Item (1) that 
\begin{equation}\label{disc:n_m}
\# \left\{(n,k)\left|0 \leq n < n_m, x_n \in J\right.\right\}-n_m\lambda(J)\geq -\delta_p \log n_m\geq -\delta_p \log N,
\end{equation}
with a fixed positive constant $\delta_p$ depending only on $p$.

Finally, for an appropriate lower bound of 
$$\# \left\{0 \leq n < N\left|x_n \in J\right.\right\}-N\lambda(J)$$
we face the task of estimating $\sum^M_{l=0}  q(l)\cdot p^m =\sum^M_{l=0}\mathcal{A}(l)$ from below: 
We know that for each $l$ in $\{0,\ldots,M\}$, for at least $\mathcal{A}(l)$ values of $k$s with $k+w_l<p^m$, the term 
$$\xi_{w_{l}}\cdot \left(\begin{tabular}{c} $u_0$ \\ $\vdots$ \\ $u_{w_{l}-1}$ \end{tabular}\right)+ c^{(\boldsymbol{0}^{(w_{l})},\boldsymbol{1}^{(w_{l})})}_{k+w_{l},p^m-w_{l}} \left(\begin{array}{c}\overline{0}\\\overline{1}\\\vdots\\\overline{1}\\0\\\vdots\\0\end{array}\right)\quad\mbox{ or }\quad\xi_{w_{l}}\cdot \left(\begin{tabular}{c} $u_0$ \\ $\vdots$ \\ $u_{w_{l}-1}+1$ \end{tabular}\right)+ c^{(\boldsymbol{0}^{(w_{l})},\boldsymbol{1}^{(w_{l})})}_{k+w_{l},p^m-w_{l}} \left(\begin{array}{c}\overline{0}\\\overline{1}\\\vdots\\\overline{1}\\0\\\vdots\\0\end{array}\right)$$
resp. takes the same value. 

Since $\xi_{w_l}$ and $u_0,\ldots,u_{w_l-1}$ are independent of $k$, it remains to focus on 
\begin{align*}
c^{(\boldsymbol{0}^{(w_{l})},\boldsymbol{1}^{(w_{l})})}_{k+w_{l},p^m-w_{l}} \left(\begin{array}{c}\overline{0}\\\overline{1}\\\vdots\\\overline{1}\\0\\\vdots\\0\end{array}\right)&=\left(\binom{k+w_l+j}{j}\right)_{j=0,1,\ldots,p^m-w_l-1}\left(\begin{array}{c}\overline{0}\\\overline{1}\\\vdots\\\overline{1}\\0\\\vdots\\0\end{array}\right)\\
&=\sum_{j=1}^l\binom{\{(k+w_l)/p^3\}+ p^3 \lfloor(k+w_l)/p^3\rfloor p^3+jp^3}{jp^3}\\
&=\sum_{j=1}^l\binom{\lfloor(k+w_l)/p^3\rfloor+j}{j}\\
&=\binom{\lfloor(k+w_l)/p^3\rfloor+1+l}{l}-1
\end{align*}
in $\FF_p$, where we used the special setting $\boldsymbol{\eta}=\boldsymbol{0}$, $\boldsymbol{u}=\boldsymbol{1}$, $\boldsymbol{z}=\boldsymbol{0}$, Lucas Theorem and in the last step Lemma \ref{lem:1a}. 

Note that $M+1 = p^{m-7}$. For each $l\in\{0,1,\ldots,M\}$ take now those $k\in\{0,1,\ldots,p^m-1\}$ such that $\lfloor\frac{k+w_l}{p^3}\rfloor\in\{p^{m-3}-p(p^3-1)p^{m-7},\ldots,p^{m-3}-1\}$. Further note, that indeed every value between $p^{m-3}-p(p^3-1)p^{m-7}$ and $p^{m-3}-1$ is attained by $\lfloor\frac{k+w_l}{p^3}\rfloor$ for exactly $p^3$ values of $k$ between $0$ and $p^m-1$. This follows from $\lfloor\frac{p^3+w_0}{p^3}\rfloor\leq p^{m-3}-p(p^3-1)p^{m-7}$ and $\lfloor\frac{p^m-p^3+w_M}{p^3}\rfloor\geq p^{m-3}-1$. Hence, 
$$\mathcal{A}(l)\geq p^3\cdot\Big(\mbox{the number of $(p-1)$s in the $l$th column of $D_m$}\Big).$$
Note that $D_m$ is a $(p(p^3-1)p^{m-7}\times p^{m-7})$ matrix and from Lemma \ref{lem:1b} we know that $D_m$ contains $p(p^3-1)p^{m-7}\cdot p^{m-7}\left(1-\left(\frac{p+1}{2p}\right)^{m-7}\right)$ many $(p-1)$s. We choose now $m>7$ large enough such that $\left(1-\left(\frac{p+1}{2p}\right)^{m-7}\right)\geq \frac{p^5-1}{p^5}$. 

We summarize: 
\begin{align*}
\sum^M_{l=0}\mathcal{A}(l)&\geq p^3 \cdot \Big(\mbox{the number of $(p-1)$s in $D_m$}\Big)\\
&= p^3\left(1-\left(\frac{p+1}{2p}\right)^{m-7}\right) p(p^3-1)p^{m-7}\cdot p^{m-7}\\
&\geq (M+1)p^m \frac{p^3-1}{p^3}\cdot \frac{p^5-1}{p^5}.
\end{align*}

Finally, this estimate together with \eqref{eq:J_letc} and \eqref{disc:n_m} implies for $m$ large enough:  
\begin{eqnarray*}
 ND_N^* &\geq& \#  \left\{0 \leq n < N \left|x_n \in J\right.\right\}-N\lambda(J) \\
&\geq&- \delta_p \cdot \log N -p^m(M+1)\frac{2p-1}{(p-1)p^{p^3}}-p^m(M+1)\frac{p-1}{p}\\
&&  +  \sum_{i=0}^{M} \mathcal{A}(l)  \\
& \geq &  - \delta_p \cdot \log N + p^m(M+1) \cdot \left(\frac{p^3-1}{p^3}\cdot \frac{p^5-1}{p^5}-\frac{p-1}{p}-\frac{2p-1}{(p-1)p^{p^3}}\right)\\
& \geq& - \delta_p \cdot \log N + p^{2m}p^{-7} \cdot c'  \\
& \geq & c \cdot (\log N)^2.
\end{eqnarray*}
Here $c,c'$ are positive constant and we used $p^{2m} \geq \frac{\left(\log N\right)^2}{(p\log p)^2}$ and the fact that $$\left(\frac{p^3-1}{p^3}\cdot \frac{p^5-1}{p^5}-\frac{p-1}{p}-\frac{2p-1}{(p-1)p^{p^3}}\right)>0$$ for every $p\geq 2$ which can be ensured by applying basic analysis. 
\hfill$\qed$

\begin{remark}\label{rem:3}{\rm

The method of proof of Theorem \ref{thm:2} heavily uses  $\boldsymbol{\eta}=\boldsymbol{0}$, $\boldsymbol{u}=\boldsymbol{1}$, $\boldsymbol{z}=\boldsymbol{0}$, when estimating $\sum^M_{l=0}\mathcal{A}(l)$ from below. For the general setting we would have to estimate for each $l$ in $\{0,\ldots,M\}$, the number of $k$s with $k+w_l<p^m$, such that the term
$$\xi_{w_{l}}\cdot \left(\begin{tabular}{c} $u_0-z_k$ \\ $\vdots$ \\ $u_{w_{l}-1}-z_{k+w_l-1}$ \end{tabular}\right)+ c^{(\boldsymbol{\eta}^{(w_{l})},\boldsymbol{u}^{(w_{l})})}_{k+w_{l},p^m-w_{l}} \left(\begin{array}{c}\overline{0}\\\overline{1}\\\vdots\\\overline{1}\\0\\\vdots\\0\end{array}\right)+z_{k+w_l}$$
or
$$\quad\xi_{w_{l}}\cdot \left(\begin{tabular}{c} $u_0+z_k$ \\ $\vdots$ \\ $u_{w_{l}-1}+z_{k+w_l-1}+1$ \end{tabular}\right)+ c^{(\boldsymbol{\eta}^{(w_{l})},\boldsymbol{u}^{(w_{l})})}_{k+w_{l},p^m-w_{l}} \left(\begin{array}{c}\overline{0}\\\overline{1}\\\vdots\\\overline{1}\\0\\\vdots\\0\end{array}\right)+z_{k+w_l}$$
resp. takes the same value. 
Or equivalently, using Remark \ref{rem:affine}, such that 

$$\xi_{w_{l}}\cdot \left(\begin{tabular}{c} $u_0$ \\ $\vdots$ \\ $u_{w_{l}-1}$ \end{tabular}\right)+ c^{(\boldsymbol{\eta}^{(w_{l})},\boldsymbol{u}^{(w_{l})})}_{k+w_{l},p^m-w_{l}} \Big(\left(\begin{array}{c}\overline{0}\\\overline{1}\\\vdots\\\overline{1}\\0\\\vdots\\0\end{array}\right)+\left(\begin{tabular}{c} $z'_{w_l}$ \\ $\vdots$ \\ $\vdots$ \\ $\vdots$ \\ $\vdots$ \\ $z'_{p^m-1}$ \end{tabular}\right)\Big)$$
or
$$\quad\xi_{w_{l}}\cdot \left(\begin{tabular}{c} $u_0$ \\ $\vdots$ \\ $u_{w_{l}-1}+1$ \end{tabular}\right)+ c^{(\boldsymbol{\eta}^{(w_{l})},\boldsymbol{u}^{(w_{l})})}_{k+w_{l},p^m-w_{l}} \Big(\left(\begin{array}{c}\overline{0}\\\overline{1}\\\vdots\\\overline{1}\\0\\\vdots\\0\end{array}\right) +\left(\begin{tabular}{c} $z'_{w_l}$ \\ $\vdots$ \\ $\vdots$ \\ $\vdots$ \\ $\vdots$ \\ $z'_{p^m-1}$ \end{tabular}\right)\Big)$$
resp. takes the same value.  

The setting $\boldsymbol{z}\neq \boldsymbol{0}$ causes a further dependency on $k$ in both point of views $\boldsymbol{z}$ or $\boldsymbol{z}'$ resp. Then $\boldsymbol{\eta}\neq\boldsymbol{0}$ or $\boldsymbol{u}\neq \boldsymbol{1}$ prevents the simple application of Lemma \ref{lem:1a}. 
Of course there is more ``room'' in the method of proof of Theorem \ref{thm:2}. As, e.g., one degree of freedom in the proof is the choice of $w_0>w_1>\ldots >w_M$ that will cause vectors different to $(\overline{0},\overline{1},\ldots,\overline{1},0,\ldots,0)^T$, and one could try to adapt the proof for different specific setting of $\boldsymbol{\eta},\boldsymbol{u},\boldsymbol{z}$. Nevertheless, a proof for the general case seems to be technically very complicated up to impossible. 

Additionally, note that while in base $2$ the affine necklaces contain all $(2^m,2^m)$-nested perfect necklaces (see \cite{Bec}), in odd prime bases the affine necklaces form just a subset of all $(p^m,p^m)$-nested perfect necklaces. Hence a generalization of Theorem \ref{thm:2} to all normal numbers constructed via $(b^m,b^m)$-nested perfect necklaces, which we conjecture to be true, would need a much more generalized or even different method of proof. 
}
\end{remark}

\end{document}